\documentclass[12pt]{amsart}
\usepackage{graphicx,indentfirst}
\usepackage{amsmath,amssymb,mathrsfs}
\usepackage{amsthm,amscd}
\usepackage{verbatim}
\usepackage{appendix}
\usepackage{enumitem}
\usepackage{appendix}
\usepackage{amsfonts,color}
\usepackage[all]{xy}
\usepackage{syntonly}
\newtheorem{theorem}{Theorem}[section]
\newtheorem{lemma}{Lemma}[section]
\newtheorem{remark}{Remark}[section]
\usepackage[left=2.5cm,right=2.5cm,bottom=3cm,top=2cm]{geometry}

\newcommand{\abs}[1]{|#1|^2}

\usepackage{tikz}

\def\XXint#1#2#3{{\setbox0=\hbox{$#1{#2#3}{\int}$ }
\vcenter{\hbox{$#2#3$ }}\kern-.6\wd0}}

\newtheorem{prop}{Proposition}[section]
\newtheorem{defn}{Definition}[section]
\newtheorem{corr}{Corollary}[section]

\allowdisplaybreaks

\newcommand{\ric}{\mathrm{Ric}}
\newcommand{\innpro}[1]{\langle#1\rangle}
\newcommand{\bk}[1]{\Big(#1\Big)}

\newcommand{\ddbar}{i\partial\bar\partial}
\newcommand{\pe}{\varphi^{(r)}_{\epsilon,\gamma}}

\numberwithin{equation}{section}

\begin{document}
\title{On the K\"ahler Ricci flow on projective manifolds of general type}\author{Bin Guo}\address{Department of Mathematics, Rutgers University, Piscataway, NJ, 08854}\email{bguo@math.rutgers.edu}
\maketitle

\begin{abstract}
We consider the K\"ahler Ricci flow on a smooth minimal model of general type, we show that if the Ricci curvature is uniformly bounded below along the K\"ahler-Ricci flow, then the diameter is uniformly bounded. As a corollary we show that under the Ricci curvature lower bound assumption, the Gromov-Hausdorff limit of the flow is homeomorphic to the canonical model. Moreover, we can give a purely analytic proof of a recent result of Tosatti-Zhang (\cite{TZ}) that if the canonical line bundle $K_X$ is big and nef, but not ample, then the flow is of Type IIb. 
\end{abstract}

\section{Introduction}
The Ricci flow (\cite{H}) has been one of the most powerful tools in geometric analysis with remarkable applications to the study of $3$-manifolds. The complex analogue, the K\"ahler-Ricci flow, has been used by Cao (\cite{C}) to give an alternative proof of the existence of K\"ahler Einstein metrics on manifolds with negative or vanishing first Chern class (\cite{Y,Au}). Tsuji \cite{Ts} applied the K\"ahler-Ricci flow to construct a singular K\"ahler Einstein metrics on smooth minimal manifolds of general type. The analytic Minimal Model program, introduced in \cite{ST2, ST3}, aims to find the minimal model of an algebraic variety, by running the K\"ahler-Ricci flow. It is conjectured (\cite{ST3}) that the K\"ahler-Ricci flow will deform a given projective variety to its minimal model and eventually to its canonical model coupled with a canonical metric of Einstein type, in the sense of Gromov-Hausdorff.

Let $X$ be a projective $n$-dimensional manifold, with the canonical bundle $K_X$ big and nef. We consider the K\"ahler-Ricci flow
\begin{equation}\label{KRF}
\frac{\partial \omega}{\partial t}  = -\ric(\omega) - \omega, \quad \omega(0) = \omega_0,
\end{equation}
where $\omega_0$ is a K\"ahler metric on $X$. It's well-known that the equation \eqref{KRF} is equivalent to the following complex Monge-Ampere equation
\begin{equation}\label{KRFMA}
\left\{\begin{aligned}\frac{\partial\varphi}{\partial t} &= \log \frac{(\chi + e^{-t}(\omega_0 - \chi) + \ddbar \varphi)^n}{\Omega} - \varphi\\
\varphi(0) &= 0,
\end{aligned}\right.
\end{equation}
where $\Omega$ is a smooth volume form, $\chi = \ddbar\log \Omega\in c_1(K_X) = - c_1(X)$, and $\omega(t) = \chi + e^{-t}(\omega_0 - \chi) + \ddbar \varphi$. It's also well-known (\cite{Ts, TZh}) that the equation \eqref{KRFMA} has long time existence, if $K_X$ is nef. We will prove the following result:
\begin{theorem}\label{thm:1}
Let $X$ be a projective manifold with $K_X$ big and nef. If along the K\"ahler Ricci flow \eqref{KRF}, the Ricci curvature is uniformly bounded below for any $t\ge 0$, i.e., $$\ric(\omega(t))\ge -K\omega(t),$$ for some $K>0$,
then there is a constant $C >0$ such that the diameter of $(X,\omega(t))$ remain bounded, i.e., $$\mathrm{diam}(X,\omega(t))\le C.$$
\end{theorem}


\begin{remark}
If we use Kawamata's theorem (\cite{Ka}) that the nef and big canonical line bundle $K_X$ is semi-ample, by \cite{Zh, ST4} the scalar curvature along the K\"ahler Ricci flow \eqref{KRF} is uniformly bounded, hence Ricci curvature lower bound implies that Ricci curvature is uniformly bounded on both sides. Then in the proof of Theorem \ref{thm:1}, we can use Cheeger-Colding-Tian (\cite{CCT}) theory to identify the regular sets. Moreover, if $K_X$ is semi-ample and big, then the $L^\infty$ bound of $\varphi$ in \eqref{KRFMA} will simplify the proof. However, following Song's (\cite{S2}) recent  analytic proof of base point freeness for nef and big $K_X$, our proof of Theorem \ref{thm:1} does not rely on Kawamata's theorem.
\end{remark}

It is conjectured by Song-Tian in \cite{ST3} that the K\"ahler Ricci flow \eqref{KRF} will converge to the the canonical model of $X$ coupled with the unique K\"ahler Einstein current with bounded potential, in the Gromov-Hausdorff sense. Under the assumption that the Ricci curvature is uniformly bounded below, we can partially confirm this conjecture.

\begin{corr}\label{cor:1}
Under the same assumptions as Theorem \ref{thm:1}, then as $t\to\infty$,
\begin{equation*}
(X,\omega(t)) \xrightarrow{d_{GH}} (X_\infty, d_\infty),
\end{equation*}
the limit space $X_\infty$ is homeomorphic to the canonical model $X_{can}$ of $X$. Moreover, $(X_\infty, d_\infty)$ is isometric to the metric completion of $(X^\circ_{can}, g_{KE})$, where $g_{KE}$ is the unique K\"ahler-Einstein current  with bounded local potentials and $X_{can}^\circ$ is the regular part of $X_{can}$.
\end{corr}

Consider the unnormalized K\"ahler Ricci flow \begin{equation}\label{eqn:ukrf}\frac{\partial}{\partial t}\omega = -\ric(\omega), \quad \omega(0) = \omega_0,\end{equation} with long time existence. The flow \eqref{eqn:ukrf} is called to be of Type III, if $$\sup_{X\times [0,\infty)} t |Rm|(x,t)<\infty,$$ otherwise it is of Type IIb, here $|Rm|(\omega(t))$ denotes the Riemann curvature of $\omega(t)$. It's well-known that Type III condition is equivalent to the curvature is uniformly bounded along the normalized K\"ahler Ricci flow \eqref{KRF}. As a by-product of our proof of Theorem \ref{thm:1}, we obtain a purely analytic proof of a recent result of Tosatti-Zhang (\cite{TZ}), namely,
\begin{theorem}\label{TZ}\cite{TZ}
Let $X$ be a projective manifold with $K_X$ big and nef, if the K\"ahler Ricci flow \eqref{KRF} is of type III, then the canonical line bundle $K_X$ is ample.
\end{theorem}

Throughout this paper, the constants $C$ may be different from lines to lines, but they are all uniform. We also use $g$ as the associated Riemannian metric of a K\"ahler form $\omega$, for example the metric space $(X,\omega(t))$ means the space $(X,g(t))$.
\medskip

\noindent{\bf Acknowledgement}: The author would like to thank his advisor Prof. J. Song, for his constant help, support and encouragement over the years. He also wants to thank the  members of the complex geometry and PDE seminar at Columbia University, from whom he learns a lot. His thanks also goes to Prof. V. Tosatti for his careful reading of a preprint of this paper and many helpful suggestions which made this paper clearer. Finally he likes to thank Prof. X. Wang and V. Datar for their interest and helpful discussions.

\section{Preliminaries}
In this section, we will recall some definitions and theorems we will use in this note.
\begin{defn}
Let $L\to X$ be a holomorphic line bundle over a projective manifold $X$. $L$ is said to be semi-ample if the linear system $|kL|$ is base point free for some $k\in\mathbb Z^+$. $L$ is said to be big if the Iitaka dimension of $L$ is equal to the dimension of $X$. $L$ is called numerically effective (nef) if $L\cdot C\ge 0$ for any irreducible curve $C\subset X$.
\end{defn}

We can define a semi-group $$\mathcal F(X,L) = \{k\in\mathbb Z^+| kL\text{ is base point free}\}.$$ For any $k\in \mathcal F(X,L)$, the linear system $|kL|$ induces a morphism $$\Phi_k = \Phi_{|kL|}: X\to X_k = \Phi_k(X)\subset \mathbb P^{N_k}$$ where $N_k+1 = \mathrm{dim} H^0(X,kL)$. It's well-known that (\cite{L}) that for large enough $k,l\in \mathcal F(X,L)$, $(\Phi_k)_* \mathcal O_X = \mathcal O_{X_k}$, $\Phi_k = \Phi_l$ and $X_k = X_l$.

We will need the following version of $L^2$ estimates due to Demailly (\cite{De})

\begin{theorem}\label{thm:L2}
Suppose $X$ is an $n$-dimensional projective manifold equipped with a smooth K\"ahler metric $\omega$. Let $L$ be a holomorphic line bundle over $X$ equipped with a possibly singular hermitian metric $h$ such that $\ric(h) +\ric(\omega)\ge \delta \omega$ in the current sense for some $\delta>0$. Then for any $L$-valued $(0,1)$-form $\tau$ satisfying 
$$\bar\partial \tau = 0,\quad \int_X \abs{\tau}_{h,\omega}\omega^n<\infty,$$ there exists a smooth section $u$ of $L$ such that $\bar\partial u = \tau$ and $$\int_X \abs{u}_{h,\omega}\omega^n\le \frac{1}{2\pi\delta}\int_X \abs{\tau}_{h,\omega}\omega^n.$$
\end{theorem}

\section{Identify the regular sets}

Since $K_X$ is big and nef, by Kodaira lemma, there exists an effective divisor $D\subset X$, such that $K_X - \varepsilon D$ is ample, hence there exists a hermitian metric on $[D]$ such that $$\chi - \varepsilon \ric(h_D)>0.$$

Let's recall a few known  estimates of the flow, (see \cite{PSS, Ts} or \cite{S1,S2} without assuming that $K_X$ is semi-ample)
\begin{lemma}\label{lemma 1}
\begin{enumerate}[label=(\roman*)]
\item There is a constant $C>0$ such that for any $t\ge 0$, $$\sup_X\varphi(t)\le C, \quad \sup_X \dot \varphi(t)=\sup_X\frac{\partial}{\partial t}\varphi\le C.$$
\item For any $\delta\in (0,1)$, there is a constant $C_\delta>0$ such that $$\varphi\ge \delta \log \abs{\sigma_D}_{h_D} - C_\delta,$$
where $\sigma_D$ is a holomorphic section of the line bundle $[D]$ associated to the divisor $D$.
\item Along the K\"ahler Ricci flow, there exist constants $C>0,\lambda>0$ such that
$$\mathrm{tr}_{\omega_0}\omega(t)\le C|\sigma_D|^{-2\lambda}_{h_D}$$
\item For any compact subset $K\subset X\backslash D$, any $\ell\in \mathbb Z^+$, there exists a constant $C_{\ell, K}>0$ such that
$$\|\varphi\|_{C^\ell(K)}\le C_{\ell,K}.$$
\end{enumerate}
\end{lemma}

Hence we can conclude that $$\omega(t)\xrightarrow{C^\infty_{loc}(X\backslash D)} \omega_\infty,$$
for some smooth K\"ahler metric $\omega_\infty$ on $X\backslash D$. On the other hand, it can be shown that $\dot \varphi(t) \to 0$ on any $K\subset X\backslash D$ as $t\to \infty$, hence $\omega_\infty$ satisfies the equation $$\omega_\infty^n  = (\chi + \ddbar \varphi_\infty)^n = e^{\varphi_\infty}\Omega, \text{ on }X\backslash D,$$
and $\omega_\infty$ is a K\"ahler-Einstein metric on $X\backslash D$, i.e., $$\ric(\omega_\infty) = - \omega_\infty.$$

\begin{prop}\label{prop:1}\cite{S2}
For any holomorphic section $\sigma\in H^0(X,mK_X)$, there is a constant $C=C(\sigma)$ such that for any $t\ge 0$, we have
$$\sup_X |\sigma|_{h_t^m}^2 \le C, \qquad  \sup_X |\nabla_t \sigma|_{h_t^m}^2\le C$$ where $h_t = \frac{1}{\omega(t)^n}$ is the hermtian metric on $K_X$ induced by the K\"ahler metric $\omega(t)$ and $\nabla_t$ is the covariant derivative with respect to $h_t^m$.
\end{prop}
Letting $t\to \infty$, we have $h_t\to h_\infty = h_{KE} = h_\chi e^{-\varphi_\infty}$ on $X\backslash D$ (here $h_\chi = \frac{1}{\Omega}$), and
\begin{equation}\label{eqn:bd}\sup_{X\backslash D} |\sigma|_{h_\infty^m}^2\le C, \quad\sup_{X\backslash D} |\nabla_\infty \sigma|_{h_\infty^m}^2\le C.\end{equation}

\begin{defn}
We define a set $\mathcal R_X\subset X$ to be the points $p\in X$ such that the $\mu$-jets at $p$ are generated by global sections of $mK_X$ for some $m\in \mathbb Z^+$, for any $\mu\in\mathbb N^n$ with $|\mu|\le 2$.
\end{defn}

\begin{prop}\cite{S2}
$\mathcal R_X$ is an open dense set of $X$ and on $\mathcal R_X$ we have locally smooth convergence of $\omega(t)$ to $\omega_\infty$.
\end{prop}

By the smooth convergence of $\omega(t)$ on $X\backslash D$, we can choose a point $p\in X\backslash D$ and a small $r_0>0$ such that (we write the associated Riemannian metric of $\omega(t)$ as $g(t)$)  $$B_{g(t)}(p,r_0)\subset\subset X\backslash D, \qquad Vol_{g(t)}(B_{g(t)}(p,r_0))\ge v_0,\quad \forall t\ge 0$$
for some $v_0>0$. For any sequence $t_i\to\infty$, $(X,g(t_i),p)$ is a sequence of almost K\"ahler-Einstein manifolds (see the Appendix), in the sense of Tian-Wang (\cite{TiWa}). By the structure theorem in Tian-Wang (\cite{TiWa}), we have
\begin{equation}\label{eqn:GH}
(X,g(t_i),p)\xrightarrow{d_{GH}} (X_\infty, d_\infty, p_\infty).
\end{equation}
Moreover, $X_\infty$ has a regular-singular decomposition, $X_\infty = \mathcal R\cup \mathcal S$; the singular $\mathcal S$ is closed and of Hausdroff dimension $\le 2n - 4$; the regular set $\mathcal R$ is an open smooth K\"ahler manifold, and $d_\infty|_{\mathcal R}$ is induced by some smooth K\"ahler-Einstein metric $g'_\infty$, i.e. on $\mathcal R$, $\ric(g'_\infty) = - g'_\infty$. 

We define a subset $\mathcal S_X\subset X_\infty$ to be a set consisting of the points $q\in X_\infty$ such that there exist a sequence of points $q_k\in X\backslash \mathcal R_X$ such that $q_k\to q$ along the Gromov-Hausdroff convergence.

By a theorem of Rong-Zhang (see Theorem 4.1 in \cite{RZ}), there exists a surjective map $$\overline{(\mathcal R_X,g_\infty)} \to (X_\infty, d_\infty),$$ where $\overline{(\mathcal R_X,g_\infty)}$ denotes the metric completion of the metric space $(\mathcal R_X,g_\infty)$, 
and a homeomorphism $(\mathcal R_X, g_\infty)\to (X_\infty\backslash\mathcal S_X,d_\infty)$ which is a local isometry. 

It's not hard to see that $\mathcal S_X$ is closed in $X_\infty$ and any tangent cone at $q\not\in \mathcal S_X$ is $\mathbb R^{2n}$, hence $X_\infty\backslash \mathcal S_X\subset \mathcal R$, i.e., $\mathcal S\subset\mathcal S_X$. 
\begin{prop}
We have $$\mathcal S_X\subset \mathcal S$$ hence $\mathcal S = \mathcal S_X$.
\end{prop}
\begin{proof}
Suppose not, there exists $q\in \mathcal S_X\cap \mathcal R$, then there exist $q_k\in (X\backslash \mathcal R_X,g(t_k))$ converging to $q$ along the Gromov-Hausdroff convergence \eqref{eqn:GH}. Since $\mathcal R$ is open and tangent cones at points in $\mathcal R$ is the Euclidean space $\mathbb R^{2n}$, for any small $\delta>0$, there exists a sufficiently small $r_0>0$ such that
\begin{equation*}
B_{d_\infty}(q,3r_0)\subset \subset \mathcal R, \qquad Vol_{g_\infty'}(B_{d_\infty}(q,3r_0))>(1-\delta/2)Vol_{g_E}(B(0,3r_0)),
\end{equation*}
where $g_E$ is the standard Euclidean metric on $\mathbb R^{2n}$ and $B(0,3r_0)$ is the Euclidean ball.
Since Ricci curvatures are bounded below, by volume continuity for the Gromov-Hausdorff convergence (\cite{Co}) we have for $k$ large enough, 
\begin{equation*}
Vol_{g(t_k)}(B_{g(t_k)} (q_k,3r_0))>(1-\delta ) Vol_{g_E}(B(0,3r_0)).
\end{equation*}
By assumption that the Ricci curvature is uniformly bounded below along the K\"ahler Ricci flow, hence Perelman's pseudo-locality (\cite{P, TiWa}) implies that if $\delta$ is small enough, there exists a small but uniform constant $\varepsilon_0>0$ such that 
\begin{equation*}
\sup_{B_{g(t_k)}(q_k, 2r_0)} |Rm(g(t_k+\varepsilon_0))|\le \frac{2}{\varepsilon_0}.
\end{equation*}
Moreover, by Theorem 4.2 in \cite{TiWa}, \begin{equation}\label{eqn:GHp}
(B_{g(t_k)}(q_k,2r_0), g(t_k+\varepsilon_0),q_k)\xrightarrow{d_{GH}} (B_{d_\infty}(q,2r_0),d_\infty,q).
\end{equation}
By Shi's derivative estimate, we have
\begin{equation*}
\sup_{B_{g(t_k)}(q_k,3r_0/2)} |\nabla^l Rm(t_k+\varepsilon_0)|\le C(\varepsilon_0,l),
\end{equation*}
for any $l\in\mathbb N$ and some constant $C(\varepsilon_0,l)$.
Thus we have smooth convergence of $g(t_k+\varepsilon_0)$ to a K\"ahler metric $\tilde g_\infty$ on $(B_{d_\infty}(q,r_0), J_\infty)$ along the Gromov-Hausdorff convergence \eqref{eqn:GHp}, where $J_\infty$ is the limit complex structure.

Without loss of generality we can assume the injectivity radii of  $g(t_k+\varepsilon_0)$ at $q_k$ are bounded below by $r_0$ (\cite{CGT}), since the Riemann curvatures and volumes of $B_{g(t_k)}(q_k,r_0)$ are uniformly bounded. For $k$ large enough, there exists (see \cite{TY}) a local holomorphic coordinates system $\{z^{(k)}_\alpha\}_{\alpha=1}^n$ on the ball $(B_{g(t_k)}(q_k,r_0),g(t_k+\varepsilon_0))$ such that $\abs{z^{(k)}} = \sum_{\alpha=1}^n\abs{z^{(k)}_\alpha}\le r^2_0$, $\abs{z^{(k)}}(q_k) = 0$ and under these coordinates $g_{\alpha\bar \beta} = g_{t_k+\varepsilon_0}(\nabla z^{(k)}_{\alpha}, \bar\nabla z^{(k)}_\beta )$ satisfies
$$\frac 1 C \delta_{\alpha\beta}\le g_{\alpha\bar\beta}\le C \delta_{\alpha\beta},\quad \|g_{\alpha\bar\beta}\|_{C^{1,\gamma}}\le C,\quad\text{for some }\gamma\in (0,1).$$ 
This implies that the Euclidean metric under these coordinates
\begin{equation}\label{eqn:Eucl}\sum_{\alpha=1}^n\sqrt{-1}dz^{(k)}_\alpha\wedge d\bar z^{(k)}_\alpha\end{equation}
is uniformly equivalent to $g(t_k+\varepsilon_0)$ on the ball $B_{g(t_k)}(q_k,r_0)$.

Recall that along K\"ahler-Ricci flow $$\ric(\omega(t))  = -\chi - \ddbar (\varphi+\dot\varphi).$$

Take a cut-off function $\eta$ on $\mathbb R$ such that $\eta(x) = 1$ for $x\in (-\infty, 1/2)$ and vanishes for $x\in [1,\infty)$. Choose a function $$\Phi_k = (|\mu| + 1 + n)\eta\bk{\frac{\abs{z^{(k)}}}{ r_0^2/2}} \log \abs{z^{(k)}} + \varphi(t_k+\varepsilon_0) + \dot\varphi(t_k+\varepsilon_0).$$
 Note that $\Phi_k$ is a globally defined function on $X$ (with a log-pole at $q_k$) when $k$ is large enough.

Since the metrics \eqref{eqn:Eucl} and $g(t_k+\varepsilon_0)$ are uniformly equivalent for $k$ large enough on the support of $\ddbar \bk{(n+1+|\mu|)\eta\bk{\frac{\abs{z^{(k)}}}{r_0^2/2}}\log \abs{z^{(k)}}}$, we see that there is a uniform constant $\Lambda$ independent of $k$ such that
\begin{equation*}
\ddbar \bk{(n+1+|\mu|)\eta\bk{\frac{\abs{z^{(k)}}}{r_0^2/2}}\log \abs{z^{(k)}}}\ge -\Lambda \omega(t_k+\varepsilon_0).
\end{equation*}
We will fix an integer $m\ge 10\Lambda$.

Define a (singular) hermitian metric on $K_X$ by $$h_k = h_\chi e^{-\frac{\varphi(t_k+\varepsilon_0)}{2} -\frac{\epsilon}{m}\log \abs{\sigma_D}_{h_D} },$$ for some small $\epsilon>0$. 
Then we have for $k$ large enough (we denote below $\omega_k = \omega(t_k+\varepsilon_0)$,  and $[D]$ the current of integration over the divisor $D$.)
\begin{align*}
\ric(h_k^m) + \ric(\omega_k) + \ddbar \Phi_k =& m\chi + \frac{1}{2}\ddbar\varphi  - \epsilon\ric(h_D) + \epsilon [D] - \chi \\
& + \ddbar \bk{(n+1+|\mu|)\eta\bk{\frac{\abs{z^{(k)}}}{r_0^2/2}}\log \abs{z^{(k)}}}\\
=& \frac{m}{2}\omega_k + \frac{m\chi}{2} - \epsilon \ric(h_D)- \frac{m}{2}e^{-t_k - \varepsilon_0}(\omega_0 - \chi)\\
& + \epsilon[D] + \ddbar \bk{(n+1+|\mu|)\eta\bk{\frac{\abs{z^{(k)}}}{r_0^2/2}}\log \abs{z^{(k)}}}\\
\ge & \frac{m}{4}\omega_k,
\end{align*}
in the current sense, for $k$ large enough. The above inequality follows since $\frac{m}{2}\chi - \epsilon\ric(h_D)$ is a fixed K\"ahler metric, which is greater than $\frac{m}{2}e^{-t_k - \varepsilon_0}(\omega_0 - \chi)$ for $k$ large enough.

Define an $mK_X$-valued $(0,1)$ form $$\eta_{k,\mu} = \bar \partial \bk{ \eta\bk{\frac{\abs{z^{(k)}}}{ r_0^2/2}} (z^{(k)})^\mu },$$
where $$(z^{(k)})^\mu = \prod_{\alpha = 1}^n (z^{(k)}_\alpha)^{\mu_\alpha}, \quad \mu = (\mu_1,\ldots, \mu_n)\in\mathbb N^n.$$
It's not hard to see (noting that the pole order along $D$ is $\le \epsilon$)
\begin{equation*}
\int_X |\eta_{k,\mu}|^2_{h_k^m} e^{-\Phi_k}\omega_k^n<\infty.
\end{equation*}
Then we can apply the Hormander's $L^2$ estimate (see Theorem \ref{thm:L2} with $L = mK_X$) to solve the following $\bar\partial$-equation
\begin{equation*}
\bar\partial u_{k,\mu} = \eta_{k,\mu}, 
\end{equation*}
with $u_{k,\mu}$ a smooth section of $mK_X$ satisfying
\begin{equation*}
\int_X |u_{k,\mu}|_{h_k^m} e^{-\Phi_k}\omega_k^n\le \frac{4}{m}\int_X |\eta_{k,\mu}|_{h_k^m}^2 e^{-\Phi_k}\omega_k^n<\infty.
\end{equation*}
By checking the pole order of $e^{-\Phi_k}$ at $q_k$ we can see that $u_{k,\mu}$ vanishes at $q_k$ up to order $|\mu|$, and hence $$\sigma_{k,\mu} := u_{k,\mu} - \eta\bk{\frac{\abs{z^{(k)}}}{ r_0^2/2}} (z^{(k)})^\mu$$
is a nontrivial global holomorphic section of $mK_X$. Hence we see the global sections of $mK_X$ generates the $\mu$-jets at $q_k$ for $k$ large enough. This gives the contradiction. Hence $\mathcal S_X\subset\mathcal S$.
\end{proof}
Thus we have a local isometry homeomorphism 
$$(\mathcal R_X, g_\infty)\to (X_\infty\backslash \mathcal S_X, d_\infty) = (\mathcal R, d_\infty).$$
Hence we can identify $\mathcal R_X$ and $\mathcal R$, and $d_\infty|_{\mathcal R}$ is induced by the K\"ahler-Einstein metric $g_\infty|_{\mathcal R_X}$.

 \section{Estimates near the singular set}
Throughout this section, we fix an effective divisor $D\subset X$ such that $$K_X - \varepsilon [D]>0$$
for sufficiently small $\varepsilon>0$. By the previous section, we see $X\backslash D\subset\mathcal R_X$. Choose a log-resolution of $(X,D)$, $$\pi_1: Z\to X$$ such that $\pi_1^{-1}(D)$ is a smooth divisor with simple normal crossings. Fix a point $O$ in a smooth component of $\pi_1^{-1}(D)$ and blow up $Z$ at the point $O$, we get a map $$\pi_2: \tilde X\to Z,$$
for some smooth projective manifold $\tilde X$. Denote $\pi = \pi_1\circ \pi_2: \tilde X\to X$.

By Adjunction formula, we have
$$K_{\tilde X} = \pi^* K_X + (n-1) E + F,\qquad F = \sum_k a_k F_k,$$
where $E$ is the exceptional locus of the blow up $\pi_2$, and $F_k$ is a prime divisor in the exceptional locus of $\pi$. We also note that $a_k>0$ for any $k$.

Since $\tilde \chi = \pi^* \chi\in \pi^*K_X$ is big and nef, Kodaira's lemma implies there exists an effective divisor $\tilde D$ whose support coincide with the exceptional locus $E, F$ and $$\tilde \chi - \varepsilon [\tilde D]\text{ is K\"ahler},$$
hence there exists a hermitian metric $h_{\tilde D}$ on the line bundle associated to $D$ such that $$\tilde\chi - \varepsilon  \ric(h_{\tilde D})>0.$$
We write $\tilde D = \tilde D' + \tilde D''$, where $\mathrm{supp} D'' = E$, and $E\not\subset \tilde D'$. Let $\sigma_E$, $\sigma_F$, $\sigma_{\tilde D}$ be the defining section of $E, F$ and $\tilde D$, respectively. Here these sections are multi-valued holomorphic sections which become global after taking some power. There also exist hermtian metrics $h_E$, $h_F$, and $h_{\tilde D}$ such that $$\pi^*\Omega = |\sigma_E|_{h_E}^{2(n-1)} |\sigma_F|_{h_F}^2 \tilde \Omega,$$
for some smooth volume form $\tilde \Omega$ on $\tilde X$.

We fix a K\"ahler metric $\tilde \omega$ on $\tilde X$. The K\"ahler Ricci flow on $X$ is pulled back to $\tilde X$ by the map $\pi$, and it safeties the equation
\begin{equation}\label{pKRF}
\frac{\partial}{\partial t} \pi^*\varphi = \log \frac{(\tilde \chi + e^{-t}(\pi^*\omega_0 - \tilde \chi) + \ddbar \pi^*\varphi)^n}{|\sigma_E|_{h_E}^{2(n-1)} |\sigma_F|_{h_F}^2 \tilde \Omega} - \pi^*\varphi,
\end{equation}
with the initial $\pi^*\varphi(0) = 0$. By the previous estimates, we see that $\pi^*\varphi$ satisfies the estimates
\begin{equation*}
\delta \log |\sigma_{\tilde D}|^2_{h_{\tilde D}} - C_\delta \le \pi^*\varphi(t)\le C, \forall t\ge 0 \text{ and  }\forall \delta\in (0,1). 
\end{equation*}

We will consider a family of perturbed parabolic Monge-Ampere equations for $\epsilon\in (0,1)$
\begin{equation}\label{KRF2}
\left\{\begin{aligned}
\frac{\partial}{\partial t} \tilde\varphi_\epsilon &= \log \frac{(\tilde \chi + e^{-t}(\pi^*\omega_0 - \tilde \chi)+\epsilon \tilde\omega + \ddbar \pi^*\varphi)^n}{(|\sigma_E|_{h_E}^{2(n-1)}+\epsilon)( |\sigma_F|_{h_F}^2+\epsilon) \tilde \Omega} -\tilde \varphi_\epsilon,\\
\tilde \varphi_\epsilon(0)&= 0
\end{aligned}\right.
\end{equation}
where $\tilde \varphi_\epsilon(t)\in PSH(\tilde X, \tilde \chi + e^{-t}(\pi^*\omega_0 - \tilde\chi)+\epsilon\tilde \omega) $. The equation \eqref{KRF2} has long time existence \cite{TZh}, and we will show that solutions to \eqref{KRF2} converge to that of \eqref{pKRF} in some sense. 

It's easy to check that the K\"ahler metrics $\tilde \omega_\epsilon = \tilde \chi + e^{-t}(\pi^*\omega_0 - \tilde\chi)+\epsilon\tilde \omega + \ddbar \tilde \varphi_\epsilon$ satisfies the following evolution equation
\begin{equation}
\frac{\partial}{\partial t} \tilde \omega_\epsilon = -\ric(\tilde\omega_\epsilon) - \tilde \omega_\epsilon+\tilde \chi +\epsilon\tilde \omega - \ddbar \log \bk{(|\sigma_E|_{h_E}^{2(n-1)} + \epsilon) (|\sigma_F|_{h_F}^2 + \epsilon)\tilde \Omega}.
\end{equation}
By direct calculations, for any smooth nonnegative function $f$, we have
\begin{align*}
\ddbar \log (\epsilon+f) &= \frac{\ddbar f}{\epsilon + f} - \frac{\partial f\wedge \bar\partial f}{(f+\epsilon)^2 }\\
& = \frac{f}{f+\epsilon}\ddbar \log f + \frac{\epsilon}{f(f+\epsilon)^2}\partial f\wedge \bar\partial f\\
&\ge \frac{f}{f+\epsilon}\ddbar \log f, 
\end{align*}
in the smooth sense on $\tilde X\backslash \{f=0\}$ and globally as currents.
So 
\begin{align*}
&\ddbar \log \bk{(|\sigma_E|_{h_E}^{2(n-1)} + \epsilon) (|\sigma_F|_{h_F}^2 + \epsilon)\tilde \Omega}\\
&\ge -\frac{(n-1)|\sigma_E|_{h_E}^{2(n-1)}}{|\sigma_E|_{h_E}^{2(n-1)} + \epsilon} \ric(h_E) - \frac{|\sigma_F|_{h_F}^2}{|\sigma_F|_{h_F}^2 + \epsilon} \ric(h_F) - \ric(\tilde \Omega)  \\
&\ge -C \tilde \omega,
\end{align*}
for some uniform constant $C$ independent of $\epsilon$. Thus away from supp$\tilde D=\mathrm{supp}E\cup \mathrm{supp}F$, we have
\begin{equation}\label{met}
\frac{\partial}{\partial t}\tilde \omega_\epsilon\le -\ric(\tilde \omega_\epsilon) - \tilde \omega_\epsilon + C\tilde \omega.
\end{equation}

\begin{lemma}\label{lemma 4.1}
Let $\tilde\varphi_\epsilon$ be the solution to \eqref{KRF2}, then there exists a constant $C>0$ such that for any $t\ge 0$, $\epsilon\in (0,1)$, we have
\begin{equation*}
\sup_{\tilde X} \tilde\varphi_\epsilon(t,\cdot)\le C,\qquad \sup_{\tilde X} \frac{\partial\tilde \varphi_\epsilon}{\partial t}(t,\cdot)\le C.
\end{equation*}
\end{lemma}
\begin{proof}
Let $$V_\epsilon = \int_{\tilde X} (|\sigma_E|_{h_E}^{2(n-1)} + \epsilon) (|\sigma_F|_{h_F}^2 + \epsilon)\tilde \Omega$$
be the volume with respect to the volume form $ (|\sigma_E|_{h_E}^{2(n-1)} + \epsilon) (|\sigma_F|_{h_F}^2 + \epsilon)\tilde \Omega$. We see that $$V_1\ge V_\epsilon\ge V_0 = \int_X \Omega,$$
hence $V_\epsilon$ is uniformly bounded. We consider (for simplicity we denote $\tilde \Omega_\epsilon =  (|\sigma_E|_{h_E}^{2(n-1)} + \epsilon) (|\sigma_F|_{h_F}^2 + \epsilon)\tilde \Omega$)
\begin{align*}
&\frac{\partial}{\partial t}\bk{\frac{1}{V_\epsilon}\int_{\tilde X} \tilde\varphi_\epsilon\tilde \Omega_\epsilon}\\
 =&  \frac{1}{V_\epsilon}\int_{\tilde X}\log \frac{(\tilde \chi + e^{-t}(\pi^*\omega_0 - \tilde \chi) + \epsilon\tilde \omega+\ddbar \tilde\varphi_\epsilon)^n}{ (|\sigma_E|_{h_E}^{2(n-1)} + \epsilon) (|\sigma_F|_{h_F}^2 + \epsilon)\tilde \Omega} \tilde \Omega_\epsilon - \frac{1}{V_\epsilon}\int_{\tilde X} \tilde \varphi_\epsilon \tilde \Omega_\epsilon\\
 \le & \log\bk{ \int_{\tilde X} (\tilde \chi + e^{-t}(\pi^*\omega_0 - \tilde \chi) + \epsilon\tilde\omega + \ddbar \tilde \varphi_\epsilon)^n } - \frac{1}{V_\epsilon}\int_{\tilde X} \tilde\varphi_\epsilon \tilde \Omega_\epsilon\\
 \le & C  - \frac{1}{V_\epsilon}\int_{\tilde X} \tilde\varphi_\epsilon \tilde \Omega_\epsilon,
\end{align*}
where for the first inequality we use Jensen's inequality. From the above we see that
\begin{equation*}
\frac{1}{V_\epsilon}\int_{\tilde X} \tilde \varphi_\epsilon \tilde \Omega_\epsilon\le C.
\end{equation*}
Since $\tilde \varphi_\epsilon\in PSH(\tilde X, \tilde\chi + e^{-t}(\pi^*\omega_0 - \tilde \chi) + \epsilon\tilde \omega)$, the mean value inequality implies the uniform upper bound of $\tilde \varphi_\epsilon$
$$\sup_{\tilde X}\tilde\varphi_\epsilon(t)\le C.$$
Direct calculations show that (we will denote $\dot{\tilde \varphi}_\epsilon = \frac{\partial}{\partial t}\tilde \varphi_\epsilon$)
\begin{align*}
\frac{\partial}{\partial t}\dot{\tilde \varphi}_\epsilon & = \Delta_{\tilde \omega_\epsilon}\dot{\tilde \varphi}_\epsilon - tr_{\tilde \omega_\epsilon} e^{-t} (\pi^*\omega_0 - \tilde \chi) - \dot{\tilde \varphi}_\epsilon\\
& = \Delta_{\tilde \omega_\epsilon} \dot{\tilde \varphi}_\epsilon - n + tr_{\tilde\omega_\epsilon}\tilde\chi + \epsilon tr_{\tilde\omega_\epsilon}\tilde \omega + \Delta_{\tilde\omega_\epsilon}\tilde\varphi_\epsilon - \dot{\tilde\varphi}_\epsilon.
\end{align*}
Hence
\begin{equation}
(\frac{\partial}{\partial t} - \Delta_{\tilde \omega_\epsilon})\bk{ (e^t - 1) \dot{\tilde \varphi}_\epsilon - \tilde\varphi_\epsilon} = -tr_{\tilde\omega_\epsilon}\pi^*\omega_0 + n - \epsilon tr_{\tilde \omega_\epsilon}\tilde \omega\le n,
\end{equation}
then maximum principle implies 
\begin{equation}\label{eqn:con}
\dot{\tilde \varphi}_\epsilon(t)\le \frac{\tilde\varphi_\epsilon + n t}{ e^{t} - 1}\le C, \quad\forall t>0.
\end{equation}

\end{proof}

\begin{lemma}\label{lemma 4.2}
For any $\delta\in (0,1)$, there is a constant $C=C_\delta$ such that for any $t\ge 0$, we have
\begin{equation*}
\tilde\varphi_\epsilon(t)\ge \delta \log \abs{\sigma_{\tilde D}}_{h_{\tilde D}} - C_\delta.
\end{equation*} 
\end{lemma}
\begin{proof}
We will apply the maximum principle. For any small $\delta>0$ such that
$$\tilde \chi - \delta\ric(h_{\tilde D})>0,$$ where we may also assume $\abs{\sigma_{\tilde D}}_{h_{\tilde D}}\le 1$ by rescaling the metric $h_{\tilde D}$. 
We consider the function $H:= \tilde \varphi_\epsilon - \delta \log \abs{\sigma_{\tilde D}}_{h_{\tilde D}}$. On $\tilde X\backslash \tilde D$, $H$ satisfies the equation
\begin{equation}
\frac{\partial H}{\partial t} = \log \frac{ (\tilde \chi + e^{-t}(\pi^*\omega_0 - \tilde\chi) + \epsilon \tilde\omega  -\delta \ric(h_{\tilde D}) + \ddbar H)^n }{(|\sigma_E|^{2(n-1)_{h_E}}+ \epsilon) ( |\sigma_F|_{h_F}^2 + \epsilon )\tilde \Omega} - H - \delta \log \abs{\sigma_{\tilde D}}_{h_{\tilde D}}.
\end{equation}

By choosing $\delta$ even smaller, we may obtain $\omega_\epsilon':=\tilde \chi - \delta \ric(h_{\tilde D} ) + e^{-t}(\pi^* \omega_0 - \tilde \chi ) + \epsilon\tilde \omega$ is a K\"ahler metric on $\tilde X$ for all $t\ge 0$, and these metrics are uniformly equivalent to $\tilde \omega$, i.e., there exists $C_0>0$ such that  
\begin{equation*}
C_0^{-1}\tilde \omega_\epsilon\le \omega_\epsilon' \le C_0 \tilde \omega.
\end{equation*}
Consider the Monge-Ampere equations
\begin{equation}\label{eMA}
(\tilde \chi - \delta \ric(h_{\tilde D} ) + e^{-t}(\pi^* \omega_0 - \tilde \chi ) + \epsilon\tilde \omega + \ddbar \psi_\epsilon)^n = e^{\psi_\epsilon} (|\sigma_E|_{h_E}^{2(n-1)}+\epsilon )( |\sigma_F|^2_{h_F} + \epsilon )\tilde \Omega,
\end{equation}where $\psi_\epsilon\in PSH(\tilde X, \omega_\epsilon')$. By the Aubin-Yau theorem, \eqref{eMA} admits a unique smooth solution $\psi_\epsilon$ for any $\epsilon\in (0,1)$ and $t\ge 0$. It can be seen that 
$$\frac{1}{V_\epsilon}\int_{\tilde X} e^{\psi_\epsilon} \tilde \Omega_\epsilon\le \frac{1}{V_\epsilon} \int_{\tilde X} (\omega_\epsilon ')^n\le C.$$
Hence mean value inequality implies $\sup_{\tilde X}\psi_\epsilon\le C$. Then by \cite{EGZ}, we have $\inf_{\tilde X}\psi_\epsilon\ge -C$, hence
\begin{equation*}
\|\psi_\epsilon\|_{L^\infty}\le C,
\end{equation*}
for some $C$ independent of $\epsilon\in (0,1)$ and $t\ge 0$.

Denote $\omega_\epsilon'(t) =\omega_\epsilon' + \ddbar \psi_\epsilon$. Taking derivative with respective to $t$ on both sides of \eqref{eMA}, we get
\begin{equation}\label{eMA1}
\Delta_{\omega_\epsilon'(t)}\dot\psi_\epsilon  - tr_{\omega_\epsilon'(t)} e^{-t}(\pi^*\omega_0 - \tilde \chi) = \dot\psi_\epsilon.
\end{equation}

\begin{equation*}
\Delta_{\omega_\epsilon'(t)}\psi_\epsilon = n - tr_{\omega_\epsilon'(t)} \bk{ \tilde \chi - \delta\ric(h_{\tilde D}) + e^{-t}(\pi^*\omega_0 - \tilde\chi) + \epsilon\tilde \omega},
\end{equation*}

Hence
\begin{equation*}
\Delta_{\omega_\epsilon'(t)}(\dot\psi_\epsilon - \psi_\epsilon) = \dot\psi_\epsilon  - n + tr_{\omega_\epsilon'(t)} \bk{ \tilde \chi - \delta\ric(h_{\tilde D}) + 2 e^{-t}(\pi^*\omega_0 - \tilde \chi) + \epsilon\tilde \omega }
\end{equation*}
Noting that $\tilde \chi - \delta\ric(h_{\tilde D}) + 2 e^{-t}(\pi^*\omega_0 - \tilde \chi) + \epsilon\tilde \omega>0$ is $\delta$ is chosen appropriately, hence at the maximum point of $\dot\psi_\epsilon - \psi_\epsilon$, we have $\dot\psi_\epsilon\le n$, thus
$$\dot\psi_\epsilon\le C+n\le C.$$

Let $G=H-\psi_\epsilon =\tilde \varphi_\epsilon - \delta \log \abs{\sigma_{\tilde D}}_{h_{\tilde D}}-\psi_\epsilon$. On $\tilde X\backslash \tilde D$ it satisfies the equation
\begin{equation}
\frac{\partial G}{\partial t} = \log \frac{ (\tilde \chi -\delta\ric(h_{\tilde D}) + e^{-t}(\pi^*\omega_0 - \tilde \chi) + \epsilon\tilde \omega + \ddbar \psi_\epsilon + \ddbar G )^n }{(\tilde \chi -\delta\ric(h_{\tilde D}) + e^{-t}(\pi^*\omega_0 - \tilde \chi) + \epsilon\tilde \omega + \ddbar \psi_\epsilon)^n}- G - \delta \log\abs{\sigma_{\tilde D}}_{h_{\tilde D}} - \dot\psi_\epsilon.
\end{equation}
The minimum of $G$ cannot be at $\tilde D$ for any $t\ge 0$, since it tends to $+\infty$ when approaching $\tilde D$. For any $T>0$, suppose $(p_0,t_0)\in \tilde X\backslash \tilde D\times [0,T]$ is the minimum point of $G$, then we have at this point $\frac{\partial G}{\partial t}\le 0$ and $\ddbar G\ge 0$, hence maximum principle implies at this point
\begin{equation*}
G\ge - \delta\log \abs{\sigma_{\tilde D}}_{h_{\tilde D}} - \dot\psi_\epsilon\ge -C_\delta,
\end{equation*}
combining with $L^\infty$ bound of $\psi_\epsilon$, we have
\begin{equation*}
\tilde\varphi_\epsilon\ge \delta \log \abs{\sigma_{\tilde D}}_{h_{\tilde D}} - C_\delta.
\end{equation*}

\end{proof}

\begin{lemma}\label{lemma C2}
There exist two constants $C>0,\lambda>0$ such that for all $t\ge 0$,
\begin{equation}
\tilde \omega_\epsilon\le C |\sigma_{\tilde D}|_{h_{\tilde D}}^{-2\lambda} \tilde \omega.
\end{equation}
\end{lemma}
\begin{proof}
By the classical $C^2$ estimate for Monge-Ampere equations, there is a constant $C_1$ depending on the lower bound of bisectional curvature of $\tilde \omega$, such that
\begin{equation}\label{eqn:1}
\Delta_{\tilde \omega_\epsilon}\log tr_{\tilde \omega}\tilde\omega_\epsilon\ge -  C_1 tr_{\tilde \omega_\epsilon}\tilde \omega - C_1 - \frac{tr_{\tilde \omega} \ric(\tilde\omega_\epsilon)}{tr_{\tilde \omega}\tilde \omega_\epsilon}.
\end{equation}
By the inequality \eqref{met}, we have on $\tilde X\backslash \tilde D$
\begin{equation}\label{eqn:2}
\frac{\partial}{\partial t}\log tr_{\tilde \omega}\tilde\omega_\epsilon = \frac{tr_{\tilde \omega}\frac{\partial}{\partial t} \tilde \omega_\epsilon}{tr_{\tilde \omega}\tilde\omega_\epsilon}\le \frac{-tr_{\tilde \omega}\ric(\tilde\omega_\epsilon)}{tr_{\tilde \omega}\tilde\omega_\epsilon} -1 + \frac{C}{tr_{\tilde\omega}\tilde \omega_\epsilon}.
\end{equation}
So we have on $\tilde X\backslash \tilde D$,
\begin{align*}
&(\frac{\partial}{\partial t} - \Delta_{\tilde \omega_\epsilon})\bk{ \log tr_{\tilde \omega}\tilde \omega_\epsilon - A\tilde\varphi_\epsilon + A\delta \log \abs{\sigma_{\tilde D}}_{h_{\tilde D}}}\\
\le & \frac{C}{tr_{\tilde \omega}\tilde \omega_\epsilon} + Ctr_{\tilde \omega_\epsilon}\tilde \omega - A \log \frac{\tilde\omega_\epsilon^n}{\tilde \Omega_\epsilon} + A\tilde \varphi_\epsilon +C \\
&\quad  - Atr_{\tilde \omega_\epsilon} (\tilde \chi - \delta\ric(h_{\tilde D} ) + e^{-t}(\pi^*\omega_0 - \tilde \chi) + \epsilon\tilde \omega)\\
\le & \frac{C}{tr_{\tilde \omega}\tilde \omega_\epsilon}  - 2 tr_{\tilde \omega_\epsilon}\tilde \omega -  A \log \frac{\tilde\omega_\epsilon^n}{\tilde \omega^n} - A\log \frac{\tilde \omega^n}{\tilde \Omega_\epsilon} + C\\
\le & \frac{C}{tr_{\tilde \omega}\tilde \omega_\epsilon}  - tr_{\tilde \omega_\epsilon}\tilde \omega + C,
\end{align*}
if we choose $A$ sufficiently large and $\delta$ suitably small, and in the last inequality we use the facts that $$\log \frac{\tilde \omega^n}{\tilde \Omega_\epsilon}\le C,$$
and 
\begin{equation*}
-tr_{\tilde\omega_\epsilon}\tilde \omega -A \log \frac{\tilde \omega_\epsilon^n}{\tilde \omega^n}\le -tr_{\tilde\omega_\epsilon}\tilde \omega + An\log tr_{\tilde\omega_\epsilon}\tilde \omega\le C,
\end{equation*}
since the function $x(\in\mathbb R^+)\mapsto - x + An\log x$ is bounded above.

Using the inequality \begin{equation*}
tr_{\tilde \omega}\tilde \omega_\epsilon\le (tr_{\tilde \omega_\epsilon}\tilde \omega)^{n-1}\frac{\tilde \omega^n}{\tilde \omega_\epsilon^n}\le (tr_{\tilde \omega_\epsilon}\tilde \omega)^{n-1} e^{\tilde \varphi_\epsilon + \dot{\tilde \varphi}_\epsilon} \frac{\tilde \Omega_\epsilon}{\tilde\omega^n}\le C(tr_{\tilde \omega_\epsilon}\tilde \omega)^{n-1}.
\end{equation*}
The maximum of $\log tr_{\tilde \omega}\tilde \omega_\epsilon - A\tilde\varphi_\epsilon + A\delta \log \abs{\sigma_{\tilde D}}_{h_{\tilde D}}$ cannot be at $\tilde D$, then maximum principle implies that at the maximum of $\log tr_{\tilde \omega}\tilde \omega_\epsilon - A\tilde\varphi_\epsilon + A\delta \log \abs{\sigma_{\tilde D}}_{h_{\tilde D}}$, we have
\begin{equation*}
0\le \frac{C}{tr_{\tilde \omega}\tilde\omega_\epsilon} - C(tr_{\tilde\omega}\tilde \omega_\epsilon)^{\frac{1}{n-1}} + C,
\end{equation*}
that is, at the maximum point, $$tr_{\tilde \omega}\tilde \omega_\epsilon\le C.$$ Thus we get the desired estimate.
\end{proof}
By standard Schauder estimate (\cite{PSS}) for parabolic equations, we have
\begin{lemma}\label{lemma high}
For any integer $\ell\in\mathbb Z_+$, on any compact subset $K\subset\subset \tilde X\backslash \tilde D$, there is a constant $C_{\ell, K}$ such that for any $t\ge 0$
\begin{equation*}
\|\tilde \varphi_\epsilon\|_{C^\ell(K)}\le C_{\ell,K}.
\end{equation*}
\end{lemma}

From Lemma \ref{lemma high}, we see that $\tilde\varphi_\epsilon(t)$ converge to a smooth function $\varphi_\infty$ on $\tilde X\backslash \tilde D$ as $t\to\infty$ and $\epsilon\to 0$, which satisfies the estimates
$$\delta\log\abs{\sigma_{\tilde D}}_{h_{\tilde D}} - C_\delta \le \varphi_\infty\le C, \text{ on }\tilde X\backslash \tilde D,\text{ for any }\delta\in(0,1)$$
$$\tilde\omega_\infty := \tilde \chi + \ddbar \varphi_\infty\le C |\sigma_{\tilde D}|_{h_{\tilde D}}^{-2\lambda}.$$
Moreover, from \eqref{eqn:con} we see that
\begin{equation*}
\frac{\partial}{\partial t} (\tilde \varphi_\epsilon + C e^{-t/2})\le 0.
\end{equation*}
Hence on any compact subset $K\subset \tilde X\backslash \tilde D$, the function $\tilde \varphi_\epsilon (t) + C e^{-t/2}$ decreases to a function $\varphi_{\infty,\epsilon}$ as $t\to \infty$. Hence $\dot{\tilde \varphi}_\epsilon(t)|_K$ approaches zero as $t\to \infty$ and $\epsilon\to 0$. Thus the metric $\tilde \omega_\infty$ satisfies the equation 
\begin{equation*}
\tilde \omega_\infty^n = (\tilde \chi + \ddbar \varphi_\infty)^n  = e^{\varphi_\infty} |\sigma_{E}|_{h_E}^{2(n-1)}|\sigma_F|_{h_F}^2 \tilde \Omega, \text{ on }\tilde X\backslash \tilde D.
\end{equation*}

Let $\epsilon\to 0$, $\tilde\varphi_\epsilon(t)$ tends to a function $\tilde\varphi_0(t)\in PSH(\tilde X, \tilde \chi + e^{-t}(\pi^*\omega_0 - \tilde\chi))\cap C^\infty (\tilde X\backslash \tilde D)$ in $C^\infty_{loc}(\tilde X\backslash \tilde D\times [0,\infty))$-topology, which satisfies the degenerate parabolic Monge-Ampere equation
\begin{equation}\label{eqn:MA}
\left\{\begin{aligned}
\frac{\partial\tilde \varphi_0}{\partial t} &= \log \frac{(\tilde \chi + e^{-t}(\pi^*\omega_0 - \tilde \chi) + \ddbar \tilde \varphi_0)^n}{|\sigma_E|_{h_E}^{2(n-1)} |\sigma_F|_{h_F}^2 \tilde \Omega} - \tilde \varphi_0,\\
\tilde \varphi_0(0)& = 0
\end{aligned}
\right.
\end{equation}
with the estimates
\begin{equation}\label{eqn:3}
\delta\log\abs{\sigma_{\tilde D}}_{h_{\tilde D}} - C_\delta \le \tilde\varphi_0\le C, \text{ on }\tilde X\backslash \tilde D, \forall \delta\in (0,1),
\end{equation}
\begin{equation}\label{eqn:4}
\tilde\omega_0(t) := \tilde \chi + e^{-t}(\pi^*\omega_0 - \tilde \chi) + \ddbar \tilde\varphi_0\le C |\sigma_{\tilde D}|_{h_{\tilde D}}^{-2\lambda}, \text{ on }\tilde X\backslash \tilde D.
\end{equation}
When the solutions $\tilde \varphi_0$ to \eqref{eqn:MA} are in $PSH(\tilde X, \tilde \chi + e^{-t}(\pi^*\omega_0 - \tilde \chi))\cap L^\infty(\tilde X)$, the uniqueness of such solutions has been proved in \cite{ST3}. In the following, we will adapt their method to prove the uniqueness when solutions satisfy \eqref{eqn:3}, instead of global $L^\infty$-bound.
\begin{prop}\label{prop:uniqueness}
Let $\varphi'\in PSH(\tilde X,\tilde \chi + e^{-t}(\pi^*\omega_0 - \tilde \chi))\cap C^\infty(\tilde X\backslash \tilde D\times [0,\infty))$ be a solution to the equation \eqref{eqn:MA} with the estimate \eqref{eqn:3}, then $$\varphi' = \tilde \varphi_0.$$
\end{prop}
\begin{proof}
We consider the following perturbed equation
\begin{equation*}
\left\{\begin{aligned}\frac{\partial}{\partial t}\varphi_{\epsilon, \gamma} & = \log \frac{(\tilde \chi + e^{-t}(\pi^*\omega_0 - \tilde \chi) + \epsilon\tilde \omega + \ddbar \varphi_{\epsilon, \gamma})^n}{(|\sigma_{E}|_{h_E}^{2(n-1) } + \gamma) (|\sigma_F|_{h_F}^2 + \gamma) \tilde \Omega} - \varphi_{\epsilon,\gamma}\\
\varphi_{\epsilon,\gamma}(0) & = 0,
\end{aligned}\right.
\end{equation*}
for any $\epsilon\in (0,1)$, $\gamma\in (0,1)$. By similar arguments as in Lemmas \ref{lemma 4.1}, \ref{lemma 4.2}, \ref{lemma C2}, we can get the following estimates for $\varphi_{\epsilon,\gamma}$, 
\begin{equation}\label{eqn:a}
\sup_{\tilde X} \varphi_{\epsilon,\gamma}\le C, \quad \sup_{\tilde X} \dot\varphi_{\epsilon,\gamma}\le C.
\end{equation}
For any $\delta \in (0,1)$, there is a constant $C_\delta$ such that 
\begin{equation}\label{eqn:b}
\varphi_{\epsilon,\gamma}\ge \delta \log \abs{\sigma_{\tilde D}}_{h_{\tilde D}} - C_\delta.
\end{equation}
\begin{equation}\label{eqn:c}
\omega_{\epsilon,\gamma}(t): = \tilde \chi + e^{-t}(\pi^*\omega_0 - \tilde \chi) + \epsilon\tilde \omega + \ddbar \varphi_{\epsilon,\gamma} \le C |\sigma_{\tilde D}|_{h_{\tilde D}}^{-2\lambda} \tilde \omega.
\end{equation}
\begin{equation}\label{eqn:d}
\|\varphi_{\epsilon,\gamma}\|_{C^\ell(K)}\le C_{\ell,K}, \text{ for any compact }K\subset \subset \tilde X\backslash \tilde D.
\end{equation}
Moreover, by maximum principle, we have the following monotonicity properties
\begin{equation*}
\varphi_{\epsilon,\gamma_1}\ge \varphi_{\epsilon,\gamma_2}, \text{ for any }\gamma_1\le \gamma_2, \forall \epsilon\in (0,1);
\end{equation*}
\begin{equation*}
\varphi_{\epsilon_1,\gamma}\le \varphi_{\epsilon_2,\gamma}, \text{ for any  }\epsilon_1\le \epsilon_2, \forall \gamma\in (0,1).
\end{equation*}

We can define a function \begin{equation*}
\varphi_\epsilon: = (\lim_{\gamma\to 0} \varphi_{\epsilon, \gamma})^*,
\end{equation*}
where $f^* (z) = \lim_{r\to 0}\sup_{w\in B(z,r)\backslash\{z\}} f(w)$ is the upper regularization of a function. Then $\varphi_\epsilon$ satisfies the equation
\begin{equation*}
\left\{\begin{aligned}\frac{\partial}{\partial t}\varphi_{\epsilon} & = \log \frac{(\tilde \chi + e^{-t}(\pi^*\omega_0 - \tilde \chi) + \epsilon\tilde \omega + \ddbar \varphi_{\epsilon})^n}{|\sigma_{E}|_{h_E}^{2(n-1) } |\sigma_F|_{h_F}^2\tilde \Omega} - \varphi_{\epsilon}, \text{ on }\tilde X\backslash \tilde D\\
\varphi_{\epsilon}(0) & = 0,
\end{aligned}\right.
\end{equation*}
And we have the monotonicity 
\begin{equation*}
\varphi_{\epsilon_1}\le \varphi_{\epsilon_2},\text{ for any }\epsilon_1\le \epsilon_2, \text{ on }\tilde X\backslash \tilde D.
\end{equation*}
So we can define $\varphi_0:=\lim_{\epsilon\to 0} \varphi_{\epsilon}$, which satisfies the equation
\begin{equation}\label{varphi0}
\left\{\begin{aligned}\frac{\partial}{\partial t}\varphi_{0} & = \log \frac{(\tilde \chi + e^{-t}(\pi^*\omega_0 - \tilde \chi)  + \ddbar \varphi_{0})^n}{|\sigma_{E}|_{h_E}^{2(n-1) } |\sigma_F|_{h_F}^2\tilde \Omega} - \varphi_{0}, \text{ on }\tilde X\backslash \tilde D\\
\varphi_{0}(0) & = 0,
\end{aligned}\right.
\end{equation}
The estimates \eqref{eqn:a}, \eqref{eqn:b}, \eqref{eqn:c} and \eqref{eqn:d} implies that $$\varphi_\epsilon \xrightarrow {C^\infty(K)} \varphi_0, \text{ as }\epsilon\to 0,$$ for any compact $K\subset\subset \tilde X\backslash \tilde D$,
and $\varphi_0$ satisfies similar estimates as in \eqref{eqn:a}, \eqref{eqn:b}, \eqref{eqn:c} and \eqref{eqn:d}.

For any $\varphi'$ as in Proposition \ref{prop:uniqueness}, define a function $$\psi: = \varphi_\epsilon - \varphi' - \varepsilon_0\epsilon \log \abs{\sigma_{\tilde D}}_{h_{\tilde D}}$$
 on $\tilde X\backslash \tilde D$, where $\varepsilon_0$ is a small number such that $\tilde \omega - \varepsilon_0\ric(h_{\tilde D})>0$. For any $\epsilon$, $$\psi\ge \delta\log \abs{\sigma_{\tilde D}}_{h_{\tilde D}} - C_\delta - C - \varepsilon_0\epsilon \log \abs{\sigma_{\tilde D}}_{h_{\tilde D}}\to +\infty,$$
as the point approaching $\tilde D$, if $\delta$ is small enough, say, $\delta\le \varepsilon_0\epsilon/2$. Hence the minimum of $\psi(\cdot, t)$ can only be at $\tilde X\backslash \tilde D$. And on $\tilde X\backslash \tilde D$, $\psi$ satisfies the equation
\begin{equation*}
\frac{\partial\psi}{\partial t} = \log \frac{(\tilde \chi + e^{-t}(\pi^*\omega_0 - \tilde \chi ) + \ddbar \varphi'+ \epsilon (\tilde \omega - \omega_0\ric(h_{\tilde D})) +\ddbar \psi )^n}{(\tilde \chi + e^{-t}(\pi^*\omega_0 - \tilde \chi ) + \ddbar \varphi')^n} - \psi  - \varepsilon_0\epsilon\log\abs{\sigma_{\tilde D}}_{h_{\tilde D}}.
\end{equation*}
Maximum principle argument implies that $\psi_{\min}=\inf_{\tilde X\backslash \tilde D} \psi(\cdot,t)\ge 0$. (Recall we assume $\abs{\sigma_{\tilde D}}_{h_{\tilde D}}\le 1$.) Hence
\begin{equation*}
\varphi_\epsilon\ge \varphi'+ \varepsilon_0\epsilon\log \abs{\sigma_{\tilde D}}_{h_{\tilde D}}, \text{ on }\tilde X\backslash \tilde D.
\end{equation*}
On any compact $K\subset\subset \tilde X\backslash \tilde D$, letting $\epsilon\to 0$, we get
\begin{equation*}
\varphi_0\ge \varphi',\text{ on }K,
\end{equation*}
then let $K\to \tilde X\backslash \tilde D$, we see that \begin{equation}\label{one bound}\varphi_0\ge \varphi'.\end{equation} To show the uniqueness, we only need to show $\varphi_0\le \varphi'$, and this will be done by another perturbed equation, as Song-Tian do in \cite{ST3}.

\begin{equation*}
\left\{\begin{aligned}\frac{\partial}{\partial t}\varphi^{(r)}_{\epsilon, \gamma} & = \log \frac{(\tilde \chi + e^{-t}((1-r)\pi^*\omega_0 - \tilde \chi) + \epsilon\tilde \omega + \ddbar \varphi_{\epsilon, \gamma})^n}{(|\sigma_{E}|_{h_E}^{2(n-1) } + \gamma) (|\sigma_F|_{h_F}^2 + \gamma) \tilde \Omega} - \varphi^{(r)}_{\epsilon,\gamma}\\
\varphi^{(r)}_{\epsilon,\gamma}(0) & = 0,
\end{aligned}\right.
\end{equation*}
It's not hard to see that $\varphi^{(r)}_{\epsilon,\gamma}\to \varphi_{\epsilon,\gamma}$ as $r\to 0$. Denote $\hat\omega = \tilde \chi + e^{-t}((1-r)\pi^*\omega_0 - \tilde \chi) + \epsilon\tilde \omega + \ddbar \varphi_{\epsilon, \gamma}$.
\begin{lemma}\label{lemma:1} For some constant $C>0$, we have
\begin{equation*}
\sup_{\tilde X} \varphi^{(r)}_{\epsilon,\gamma}\le C,\quad C\log \abs{\sigma_{\tilde D}}_{h_{\tilde D}} - C\le \frac{\partial}{\partial r} \varphi^{(r)} _{\epsilon,\gamma}\le 0
\end{equation*}
\end{lemma}
\begin{proof}
The upper bound of $\varphi^{(r)}_{\epsilon,\gamma}$ follows similarly as the proof in Lemma \ref{lemma 4.1}.
\begin{equation*}
\frac{\partial}{\partial t}\bk{ \frac{\partial \varphi^{(r)}_{\epsilon,\gamma}}{\partial r} } = \Delta_{\hat\omega}\frac{\partial \varphi^{(r)}_{\epsilon,\gamma}}{\partial r} - e^{-t}tr_{\hat \omega}\pi^*\omega_0 - \frac{\partial \varphi^{(r)}_{\epsilon,\gamma}}{\partial r}\le \Delta _{\hat \omega} \frac{\partial \varphi^{(r)}_{\epsilon,\gamma}}{\partial r} - \frac{\partial \varphi^{(r)}_{\epsilon,\gamma}}{\partial r}.
\end{equation*}
Maximum principle argument implies $\frac{\partial \varphi^{(r)}_{\epsilon,\gamma}}{\partial r}\le 0$.

Let $H:=\frac{\partial \varphi^{(r)}_{\epsilon,\gamma}}{\partial r} + A\pe - A\varepsilon_0\log \abs{\sigma_{\tilde D}}_{h_{\tilde D}}$, where $\varepsilon_0>0$ is a small number such that $$\tilde \chi + e^{-t}((1-r)\pi^*\omega_0 - \tilde \chi) + \epsilon\tilde \omega - \varepsilon_0\ric(h_{\tilde D})\ge c_0\tilde \omega, $$ for all $t\ge 0$ and $c_0>0$ is a uniform constant.
 
On $\tilde X\backslash \tilde D$, if we choose $A$ sufficiently large, we have
\begin{align*}
(\frac{\partial}{\partial t} - \Delta_{\hat\omega}) H =& - e^{-t} tr_{\hat \omega} \pi^*\omega_0 - \frac{\partial\pe}{\partial r} + A \log\frac{\hat \omega^n }{\tilde \Omega_\gamma} - A\pe - An\\
& + Atr_{\hat\omega} (\tilde \chi + e^{-t}((1-r)\pi^*\omega_0 - \tilde \chi) + \epsilon\tilde \omega - \varepsilon_0\ric(h_{\tilde D}))\\
\ge & - H - A\varepsilon_0\log \abs{\sigma_{\tilde D}}_{h_{\tilde D}} - C\\
\ge & - H - C.
\end{align*}
Since the minimum of $H$ cannot occur at $\tilde D$, maximum principle argument implies that $H\ge -C$, combing with the uniform upper bound of $\pe$, we conclude that $$\frac{\partial}{\partial r}\pe\ge C\log\abs{\sigma_{\tilde D}}_{h_{\tilde D}} - C.$$

\end{proof}

Let \begin{equation*}
\varphi^{(r)}_\epsilon: = (\lim_{\gamma\to 0}\pe)^*,
\end{equation*}
then it satisfies the equation
\begin{equation*}
\left\{
\begin{aligned}
\frac{\partial}{\partial t}\varphi^{(r)}_\epsilon & = \log \frac{ (\tilde \chi + e^{-t}((1-r)\pi^*\omega_0 -\tilde \chi) +\epsilon\tilde\omega + \ddbar \varphi^{(r)}_\epsilon}{|\sigma_E|_{h_E}^{2(n-1)} |\sigma_F|_{h_F}^2 \tilde \Omega} - \varphi^{(r)}_\epsilon, \text{ on }\tilde X\backslash \tilde D\\
\varphi^{(r)}_\epsilon(0) &= 0.
\end{aligned}
\right.
\end{equation*}
We have the monotonicity $\varphi^{(r)}_{\epsilon_1}\le\varphi^{(r)}_{\epsilon_2}$ for any $\epsilon_1\le \epsilon_2$. Define $$\varphi^{(r)} = \lim_{\epsilon\to 0}\varphi^{(r)}_\epsilon.$$ From Lemma \ref{lemma:1} it's not hard to see that
\begin{equation*}
|\varphi^{(r_1)} - \varphi^{(r_2)}|\le C(1-\log\abs{\sigma_{\tilde D}}_{h_{\tilde D}}) |r_1 - r_2|, \text{ on }\tilde X\backslash \tilde D,
\end{equation*}
hence on any compact subset $K\subset\subset \tilde X\backslash \tilde D$, $\varphi^{(r)}\to \varphi_0$ in the $C^\infty$ sense as $r\to 0$, where $\varphi_0$ is the solution constructed in \eqref{varphi0}. 

\newcommand{\sg}{\abs{\sigma_{\tilde D}}_{h_{\tilde D}}}

Now we are ready to finish the proof of Proposition \ref{prop:uniqueness}. Define $G := \varphi' - \varphi^{(r)} - e^{-t}r\varepsilon_0 \log \abs{\sigma_{\tilde D}}_{h_{\tilde D}}$. By the assumption on $\varphi'$,  for any fixed $t\ge 0$, $r\in (0,1)$
\begin{equation*}
G\ge \delta\log\sg - C_\delta - C - e^{-t}r\varepsilon_0\log \sg\to +\infty,
\end{equation*}
as approaching $\tilde D$, if $\delta$ is smaller than $e^{-t} r\varepsilon_0$, hence the minimum of $G$ cannot be at $\tilde D$. On the other hand, on $\tilde X\backslash \tilde D$, we have
\begin{align*}
\frac{\partial}{\partial t} G & = \log \frac{ (\tilde \chi + e^{-t}((1-r)\pi^*\omega_0 - \tilde \chi) + \ddbar \varphi^{(r)} + r e^{-t} (\pi^*\omega_0 - \varepsilon_0\ric(h_{\tilde D})) + \ddbar G)^n }{(\tilde \chi + e^{-t}((1-r)\pi^*\omega_0 - \tilde \chi) + \ddbar \varphi^{(r)})^n}  - G\\
&\ge \log \frac{ (\tilde \chi + e^{-t}((1-r)\pi^*\omega_0 - \tilde \chi) + \ddbar \varphi^{(r)} + \ddbar G)^n }{(\tilde \chi + e^{-t}((1-r)\pi^*\omega_0 - \tilde \chi) + \ddbar \varphi^{(r)})^n}  - G,
\end{align*}
by maximum principle, we have $G\ge 0$, i.e., \begin{equation*}
\varphi'\ge \varphi^{(r)} + e^{-t}r\varepsilon_0\log \sg.
\end{equation*}
On any compact subset $K\subset\subset \tilde X\backslash \tilde D$, letting $r\to 0$, we get
$$\varphi'\ge \varphi_0,\quad\text{ on }K.$$
Then let $K\to \tilde X\backslash \tilde D$, we see that $\varphi'\ge \varphi_0$ on $\tilde X\backslash \tilde D$,  combing with \eqref{one bound}, we show that $\varphi' = \varphi_0$. Hence we finish the proof of uniqueness of solutions.

\end{proof}

From the uniqueness of solutions to \eqref{eqn:MA} and estimates of $\pi^*\varphi$, we see that 
\begin{equation}\label{equnique}
\tilde \omega_\epsilon(t)\xrightarrow{C^\infty_{loc}(\tilde X\backslash \tilde D)} \pi^*\omega(t), \text{ as }\epsilon\to 0,
\end{equation}
where $\omega(t)$ is the solution to the K\"ahler Ricci flow \eqref{KRF} on $X$.

We will come back to equation \eqref{KRF2}. 

Let $O\in B_O\subset Z$ be a small Euclidean ball, $\tilde B_O = \pi_2^{-1}(B_O)\subset \tilde X$. The divisors $\tilde D'$ and $\pi^{-1}(D) -  E$ (the proper transform of $D$) lie in the zero set of of a local holomorphic function $w$ in $\tilde B_O$. By Lemma \ref{lemma C2}, we have
\begin{lemma}\label{lem:bound}
\begin{equation*}
\tilde \omega_\epsilon(t)\le \frac{C}{|w|^{2\lambda}}\tilde \omega,\quad\text{ on }\partial \tilde B_O.
\end{equation*}
\end{lemma}
Let $\hat \omega : =\pi_2^*\omega_{Eucl}$, where $\omega_{Eucl}$ is the Euclidean metric on $B_O$, then local calculation shows that (see \cite{SW,S1})
\begin{equation}\label{eqn:aa}
C_0^{-1}\hat\omega\le \tilde \omega\le \frac{C_0}{\abs{\sigma_E}_{h_E}} \hat\omega, \quad\text{in }\tilde B_O,
\end{equation}
and
\begin{equation*}
\tilde \chi - \varepsilon_0\ric(h_E)>0,\quad \text{ in }\tilde B_O.
\end{equation*}

\begin{prop}
There exist a small $\delta\in (0,1)$ and $\lambda>0$ such that for any $t\ge 0,\epsilon>0$, we have
\begin{equation}\label{equation 100}
\tilde \omega_\epsilon(t)\le \frac{C}{|\sigma_E|_{h_E}^{2(1-\delta)}|w|^{2\lambda}} \tilde \omega, \quad\text{in }\tilde B_O. 
\end{equation}
\end{prop}
\begin{proof}
We will do the calculation in $\tilde B_O\backslash E\cup\{w=0\}$. Since $\hat\omega$ has flat curvature in $\tilde B_O\backslash E\cup\{w=0\}$, we have
\begin{equation*}
\Delta_{\tilde \omega_\epsilon(t)}\log tr_{\hat\omega}\tilde\omega_\epsilon(t)\ge - \frac{tr_{\hat\omega} \ric(\tilde \omega_\epsilon(t))}{tr_{\hat\omega} \tilde \omega_\epsilon(t)},
\end{equation*}
and
\newcommand{\omt}{\tilde \omega_\epsilon(t)}
by \eqref{met}
\begin{equation*}
\frac{\partial}{\partial t}\log tr_{\hat \omega}\tilde\omega_\epsilon(t) \le \frac{ tr_{\hat\omega}(-\ric(\tilde \omega_\epsilon(t) ) - \tilde \omega_\epsilon(t) + C\tilde \omega) }{tr_{\hat\omega}\tilde\omega_\epsilon(t)}.
\end{equation*}
So
\begin{equation*}
(\frac{\partial}{\partial t} - \Delta_{\omt})\log tr_{\hat\omega}\omt\le -1 + C \frac{tr_{\hat\omega} \tilde \omega}{tr_{\hat\omega}\omt}\le \frac{C}{|\sigma_E|_{h_E}^2 tr_{\hat \omega} \omt},
\end{equation*}
where we have used \eqref{eqn:aa}.

So we have ($r$ is a sufficiently small number)
\begin{equation*}
(\frac{\partial}{\partial t} - \Delta_{\omt})\log ( |\sigma_E|_{h_E}^{2(1+r)} |w|^{2\lambda} tr_{\hat \omega}\omt )\le \frac{C}{|\sigma_E|_{h_E}^2 tr_{\hat \omega}\omt} + (1+r)tr_{\omt} \ric(h_E).
\end{equation*}

\begin{align*}
&(\frac{\partial}{\partial t} - \Delta_{\omt})\bk{\log |\sigma_E|_{h_E}^{2(1+r)} |w|^{2\lambda} tr_{\hat \omega}\omt - A\tilde \varphi_\epsilon}\\
\le & C - A\log \frac{\omt^n}{\tilde \omega^n} + \frac{C}{|\sigma_E|^2_{h_E} tr_{\hat\omega}\omt} + (1+r)\ric(h_E)\\
& \quad - A tr_{\omt}\tilde \chi - A e^{-t}tr_{\omt} (\pi^*\omega_0 - \tilde \chi) - A\epsilon tr_{\omt} \tilde \omega\\
&\le C - tr_{\omt}\tilde \omega +\frac{C}{|\sigma_E|_{h_E}^2 tr_{\hat\omega}\omt},
\end{align*}
if $A$ is sufficiently large, and in the last inequality we have used the fact that $\tilde\chi - \varepsilon \ric(h_E)$ is a K\"ahler metric on $\tilde B_O$ when $\varepsilon$ is small.

On the other hand, similar calculation shows that
\begin{equation*}
(\frac{\partial}{\partial t} - \Delta_{\omt}) \log tr_{\tilde \omega}\omt\le C_1 tr_{\omt}\tilde \omega + C_1  + \frac{C_1}{tr_{\tilde \omega}\omt},
\end{equation*}
where $C_1$ depends on the lower bound of the bisectional curvature of $\tilde \omega$. 

Define $$G = {\log |\sigma_E|_{h_E}^{2(1+r)} |w|^{2\lambda} tr_{\hat \omega}\omt - A\tilde \varphi_\epsilon} +\frac{1}{2C_1} \log |w|^{2\lambda+2} tr_{\tilde \omega}\omt,$$
by the calculations above, we have
\begin{align*}
(\frac{\partial}{\partial t} - \Delta_{\omt}) G \le & C - \frac{1}{2} tr_{\omt}\tilde \omega + \frac{C}{|\sigma_E|_{h_E}^2 tr_{\hat \omega} \omt} + \frac{1}{tr_{\tilde \omega}\omt} \\
\le & C_2 - \frac{1}{2} tr_{\omt}\tilde \omega + \frac{C}{|\sigma_E|_{h_E}^2 tr_{\hat \omega} \omt},
\end{align*}
where in the last inequality we use \eqref{eqn:aa}.

For any small positive $r$, $|\sigma_E|^{2(1+r)}_{h_E}|w|^{2\lambda}tr_{\hat\omega}\omt$ tends to $0$ as approaching $E$ and $\{w=0\}$, so for any $t\ge 0$, $G$ cannot obtain its maximum at $\tilde B_O\cap E\cap \{w=0\}$. Moreover, we know
\begin{equation*}
\tilde\varphi_\epsilon\ge \delta\log |w| - C_\delta, \text{ on }\partial \tilde B_O,
\end{equation*}
for any small $\delta>0$. Hence by Lemma \ref{lem:bound}, we have
\begin{equation*}
\sup_{\partial \tilde B_O} G\le C.
\end{equation*}
For any $T>0$, assume $(p_0,t_0)\in \overline{\tilde B_O}\backslash E\cup \{w = 0\}\times [0,T]$ is the maximum point of $G$. If $p_0\in \partial \tilde B_O$, then we are done. Otherwise, we have at this maximum point
\begin{equation*}
|\sigma_E|_{h_E}^{2} tr_{\hat\omega}\omt \big(tr_{\omt}\tilde \omega - 2C_2\big)\le C.
\end{equation*}
By the inequality
\begin{equation*}
tr_{\tilde \omega}\omt \le \frac{\omt^n}{\tilde\omega^n}(tr_{\omt}\tilde \omega)^{n-1} = (tr_{\omt}\tilde \omega)^{n-1} e^{\tilde\varphi_\epsilon+\dot{\tilde\varphi}_\epsilon} \frac{\tilde\Omega_\epsilon}{\tilde \omega^n}\le C_3 (tr_{\omt}\tilde \omega)^{n-1}.
\end{equation*}
So
\begin{equation*}
|\sigma_E|_{h_E}^{2} tr_{\hat\omega}\omt\big( (tr_{\tilde\omega}{\omt})^{1/(n-1)} - C_4 \big)\le C,\quad\text{at }(p_0,t_0).
\end{equation*}
If  $tr_{\tilde \omega}{\omt}(p_0,t_0)\le 2^{n-1} C_4^{n-1}$, then $|\sigma_E|_{h_E}^2 tr_{\hat \omega} \omt (p_0,t_0) \le 2^{n-1} C C_4^{n-1}$. Noting that in $\tilde B_O$, $$\tilde \varphi_\epsilon\ge \delta \log |\sigma_E|_{h_E}^2 + \delta \log |w|^2 - C_\delta,$$
hence $G$ is bounded above by a uniform constant, if we choose $\delta$ small enough in the above inequality.

If $tr_{\tilde \omega}{\omt}(p_0,t_0)\ge 2^{n-1} C_4^{n-1}$, then we have
\begin{equation*}
|\sigma_E|_{h_E}^{2} tr_{\hat\omega}\omt(p_0,t_0)\le C.
\end{equation*}
Then for $\delta$ small enough,
\begin{equation*}
G(p_0,t_0)\le r\log \abs{\sigma_E}_{h_E} - \delta \log \abs{\sigma_E}_{h_E} -  \delta \log |w|^2 + (\lambda+1)\log |w|^2 + C \le C.
\end{equation*}
In sum, in all cases, we have $\sup_{\tilde B_O\times [0,T]} G\le C$. Then
\begin{equation*}
\log \bk{ |\sigma_E|_{h_E}^{2(1+r)} tr_{\hat\omega} \omt  |w|^{2\lambda+(2+2\lambda)(2C_1)^{-1}}  (tr_{\tilde \omega}\omt)^{(2C_1)^{-1}}  }\le \tilde\varphi_\epsilon +C\le C,
\end{equation*}
noting that $tr_{\hat\omega}\omt\ge C_0^{-1} tr_{\tilde \omega}\omt$, we have
\begin{equation*}
 \bk{tr_{\tilde \omega} \omt}^{1+(2C_1)^{-1}}\le \frac{C}{|\sigma_E|_{h_E}^{2(1+r)} |w|^{2\Lambda}}.
\end{equation*}
If we choose $r$ sufficiently small, say, $r\le (10 C_1)^{-1}$, then $\frac{1+r}{1+ (2C_1)^{-1}} = 1-\delta$ for some $\delta\in (0,1)$,
and hence 
\begin{equation*}
tr_{\tilde \omega}\omt\le \frac{C}{|\sigma_E|_{h_E}^{2(1-\delta)} |w|^{2\Lambda}},\quad\text{in }\tilde B_O\backslash E\cup \{w= 0\}.
\end{equation*}
\end{proof}

\begin{corr}
By letting $\epsilon\to 0$ in \eqref{equation 100} and the convergence \eqref{equnique}, we have for any $t\ge 0$
\begin{equation}\label{eqn:bb}
\pi^*\omega(t)\le \frac{C}{|\sigma_E|_{h_E}^{2(1-\delta)} |w|^{2\lambda}}\tilde \omega, \text{ in }\tilde B_O\backslash (E\cup \{w=0\}).
\end{equation}
Letting $t\to\infty$, we have
\begin{equation*}
\pi^*\omega_\infty\le \frac{C}{|\sigma_E|_{h_E}^{2(1-\delta)} |w|^{2\lambda}}\tilde \omega, \text{ in }\tilde B_O\backslash (E\cup \{w=0\}).
\end{equation*}
\end{corr}

\begin{lemma}\label{finite distance}
For any $q\in D\subset X$, there exists a smooth curve $\gamma(s):\in [0,1]\to X$ such that
\begin{enumerate}[label=(\arabic*)]
\item $\gamma([0,1))\subset  X\backslash D$, and $\gamma(1) = q$;
\item $\gamma$ is transversal to $D$;
\item for any $\varepsilon>0$, there exists an $s_0>0$, such that for all $s\in [s_0, 1)$
\begin{equation*}
d_{g(t)}(q,\gamma(s))\le \varepsilon, \quad \forall t\ge 0.
\end{equation*} 
\end{enumerate}
\end{lemma}
\begin{proof}
We take the resolution $\pi_1: Z\to X$ and choose a point $O$ in a smooth component of $\pi_1^{-1}(D)$ with $\pi_1(O)=q$, and blow up $O$, $\pi_2: \tilde X\to Z$, and $\pi= \pi_1\circ\pi_2: \tilde X\to X$. We choose an appropriate smooth path $\tilde \gamma([0,1))\subset \tilde B_O\backslash E\cup \{w = 0\}$ which keeps away from $\{w=0\}$ and  $\tilde \gamma(1)\subset E$. Then  $\gamma = \pi(\tilde \gamma)$ is the desired path, and last item follows from the uniform estimate \eqref{eqn:bb}. 
\end{proof}

\begin{corr}\label{corollary finite}
For a fixed $p\in X\backslash D$, any $q\in D$, there exists a constant $C=C_q$ such that for any $t\ge 0$
\begin{equation*}
d_{g(t)}(p,q)\le C_q.
\end{equation*}
Hence along the convergent sequence $(X,g(t_i),p)\xrightarrow{d_{GH}}(X_\infty, d_\infty, p_\infty)$, $q\in (X,g(t_i))$ converges (up to a subsequence) to some $q_\infty\in X_\infty$ in the Gromov-Hausdorff sense. 
\end{corr}

Since we aim to give a purely analytic proof of our main results, without using of Kawamata's base point free theorem, we need the local freeness of some power of the canonical line bundle $K_X$ as proved in \cite{S2}, for completeness we give a sketched proof.
\begin{prop}\cite{S2}
For any $q\in D$, there exists $\sigma\in H^0(X, mK_{X})$ for some $m\in\mathbb Z_+$ such that
\begin{equation*}
\sigma(q)\neq 0.
\end{equation*}
\end{prop}
\begin{proof}
By Corollary \ref{corollary finite}, we can take $q_\infty$ as the limit point of $q$. By \cite{S2}, there exists a $\sigma\in H^0(X, mK_{X})$ such that \begin{equation*}
|\sigma|_{h_{KE}^m}^2(q_\infty)>1,
\end{equation*}
where $|\sigma|_{h_{KE}^m}^2$ is a Lipschitz continuous function on $X_\infty$. 

By Lemma \ref{finite distance}, there exists a sequence of points $\{q_k\}\subset\mathcal R_X$ which transversely tend to $q$ and 
\begin{equation*}
d_{g(t_i)}(q,q_k)\le k^{-1},\quad \forall i.
\end{equation*} 
We may assume $q_k\in\mathcal R_X$ converge to the same point $q_k\in\mathcal R = \mathcal R_X$. Hence
\begin{equation*}
d_\infty(q_\infty, q_k) = \lim_{i\to\infty} d_{g(t_i)}(q, q_k)\le k^{-1}.
\end{equation*} 
By the continuity of $|\sigma|_{h_{KE}^m}^2$, for $k$ large enough, $$e^{-m\varphi_{KE}}|\sigma|_{h_\chi^m}^2(q_k) = |\sigma|_{h_{KE}^m}^2(q_k)\ge \frac 1 2,$$
i.e. $$\varphi_{KE}(q_k)\le C +C \log \abs{\sigma}_{h_\chi^m}(q_k).$$
On the other hand, for any $\delta\in (0,1)$, we have
\begin{equation*}
\delta \log \abs{\sigma_D}_{h_D} (q_k) - C_\delta\le \varphi_{KE}(q_k),
\end{equation*}
hence
\begin{equation*}
|\sigma_D|_{h_D}^{2\delta}(q_k)\le C |\sigma|_{h_\chi^m}^2(q_k).
\end{equation*}
Since $q_k$ approaches $D$ transversely, and $\delta$ is any arbitrarily small number, we see that $\sigma$ cannot vanish at $q$. Thus complete the proof.
\end{proof}

By a compactness argument and the previous proposition, we have:

\begin{prop}\label{prop 4.4}
There exists an integer $m\in \mathbb Z_+$ such that for any $q\in X$, there exists a holomorphic section $\sigma\in H^0(X,mK_X)$ such that $\sigma(q)\neq 0$, i.e., $mK_X$ is base point free. Thus a basis $\{\sigma_0,\ldots,\sigma_{N_m}\}$ of $H^0(X,mK_X)$ gives a morphism 
\begin{equation*}
\Phi_m: X\to X_{can}\subset\mathbb{CP}^{N_m},
\end{equation*}
where $X_{can}$ is the image of $X$ under $\Phi_m$.
\end{prop}

\begin{remark}
Proposition \ref{prop 4.4} is well-known from Kawamata's base point free theorem. It follows from algebraic geometry \cite{L} that  when $mK_X$ is base point free the maps $\Phi_m$ stabilize when $m$ is sufficiently large, i.e., $\Phi_m$ is independent of $m$ when $m$ is large enough and we will denote this map by $\Phi$. 
\end{remark}

For the given basis $\{\sigma_0,\ldots,\sigma_N\}$ of $H^0(X, mK_X)$, we have
\begin{equation*}
\sum_{i=0}^N |\sigma_i|^2_{h_t^m} = \sum_{i=0}^N |\sigma_i|^2_{h_\chi^m} e^{-\varphi-\dot\varphi}\ge c_0\sum_{i=0}^N |\sigma_i|_{h_\chi^m}\ge c_1>0.
\end{equation*}
Moreover, by Proposition \ref{prop:1}, we know

\begin{equation*}
\sup_X \sum_{i=1}^N |\sigma_i|_{h_t^m}^2\le C,\quad\sup_X |\nabla_{\omega(t)}\sigma_i|_{h_t^m}^2\le C, \forall i, ~~\forall t\ge 0, 
\end{equation*}
thus the map 
\begin{equation*}
\Phi_i: (X,\omega(t_i))\to (X_{can}, \omega_{FS}), \quad x\mapsto [\sigma_0(x):\ldots:\sigma_N(x)]\in \mathbb {CP}^N 
\end{equation*}
has uniformly bounded derivatives (see e.g. \cite{DS}). Since the target space $(X_{can},\omega_{FS})$ is compact, by Arzela-Ascoli theorem, the map extends to the Gromov-Hausdorff limit space
\begin{equation*}
\Phi_\infty: (X_\infty, d_\infty)\to (X_{can}, \omega_{FS}),
\end{equation*}
which is Lipschitz continuous.

Under our assumption that the Ricci curvature is uniformly bounded below, Tian-Wang's theory (\cite{TiWa}) on the structure of limit of almost K\"ahler Einstein manifolds implies that the singular set is closed and of Hausdorff codimension at least $4$, which also implies that any tangent cone in the limit space is good (see \cite{DS}), in the sense that there exists a tangent cone $C(Y)$, such that for any $\eta>0$ there exists a cut-off function $\beta$ which is $1$ on a small neighborhood of the singular set $S_Y\subset Y$, and vanishes outside the $\eta$-neighborhood of $S_Y$, and $\|\nabla \beta\|_{L^2(Y)}\le \eta$, then following Donaldson-Sun's idea (\cite{DS}) on partial $C^0$ estimates (see also \cite{Ti}),  by similar arguments as in \cite{S1}, for any two distinct points $p,q\in X_\infty$, one can construct two holomorphic sections $\sigma_1,\sigma_2\in H^0(X_\infty,m K_{X_\infty})$ which separate $p,q$, hence we have 

\begin{prop}\cite{S1}
$\Phi_\infty$ is injective.
\end{prop}

\section{Proof of Theorems}
\begin{proof}[Proof of Theorem \ref{thm:1}]
To prove Theorem \ref{thm:1}, we will argue by contradiction. Following the ideas in  \cite{S2}, we need the following lemma:
\begin{lemma}\label{cont lemma}
Suppose $\mathrm{diam}(X,g(t_i))\to \infty$, then we have $\mathrm{diam}(X_\infty, d_\infty) = \infty$, 
\begin{enumerate}[label=(\arabic*)]
\item $\Phi_\infty: (X_\infty, d_\infty)\to (X_{can},\omega_{FS})$ is not surjective;
\item For $p\in X_{can}\backslash \Phi_\infty(X_\infty)$ and any sequence of points $q_j\in X_\infty$ with $d_{\omega_{FS}} (\Phi_\infty(q_j), p)\to 0$ as $j\to \infty$, we have
$$d_\infty(p_\infty, q_j)\to \infty.$$
\end{enumerate}
\end{lemma}
\begin{proof}
(1) Suppose $\Phi_\infty$ is surjective. Since $\mathrm{diam}(X_\infty, d_\infty) = \infty$, there exists a sequence of points $q_j\in \mathcal R\subset X_\infty$ with $d_\infty(p_\infty, q_j)\to \infty$. $(X_{can},\omega_{FS})$ is a compact metric space, hence there exists a convergent subsequence of $\{\Phi_\infty(q_j)\}$ which converge to some $q'_\infty\in X_{can}$ with respect to the metric $\omega_{FS}$. Then there is a point $q_\infty\in X_\infty$ such that $\Phi_\infty(q_\infty) = q'_\infty$. We claim that the ball $B_{d_\infty}(q_\infty,1)$ contains all but finitely many $q_j$'s. Assuming this claim, we get that the distance of $q_j$ and $p_\infty$ is bounded by $d_\infty(p_\infty,q_\infty)+1$ contradicting the choice of $q_j$ which converge to $\infty$ under $d_\infty$ as $j\to\infty$. To see the claim, suppose not, there exists a subsequence $q_{j_l}\subset \{q_j\}$ such that $d_\infty(x_{j_l},q_{j_l})\ge 1$, where $x_j$ is a sequence of points contained in  $ \mathcal R$ which converge to $q_\infty$ under the metric $d_\infty$. Since $\Phi_\infty(q_{j_l})$ and $\Phi_\infty(x_{j_l})$ are both in $X_{can}^{reg}$ which is connected and these points both converge to $q_\infty'$, so we can choose a curve $\gamma_{j_{l}}\subset X_{can}^{reg}$ whose length under $\omega_{FS}$ tend to $0$ as $j_l\to\infty$. Then $\Phi_\infty^{-1}(\gamma_{j_l})\subset \mathcal R\subset X_\infty$ is a connected curve connecting $x_{j_l}$ and $q_{j_l}$ which has  $d_\infty$-length greater than $1$, so we can take a point $y_{j_l}\in \Phi_\infty^{-1}(\gamma_{j_l})$ such that $1/2\le d_\infty(x_{j_l},y_{j_l})\le 1$. Then by compactness we may assume that up to a subsequence $y_{j_l}$ converge to a point $y_\infty\in X_\infty$ which satisfies $1/2\le d_\infty(q_\infty,y_\infty)\le 1$. It's not hard to see by triangle inequality that $d_{\omega_{FS}}(\Phi_\infty(y_{j_l}), q_\infty')\to 0$, hence $d_{\omega_{FS}}(\Phi_\infty(y_\infty),q_\infty') = 0$ and $\Phi_\infty(y_\infty) = q_\infty'$, and this contradicts the property that $\Phi_\infty$ is injective. Hence we prove the claim.


(2) Suppose $d_\infty(p_\infty, q_j)\le A$ for some constant $A>0$. By compactness we can assume a subsequence of $q_j$ converges to a point $q_\infty$, with $d_\infty(p_\infty,q_\infty)\le A$.

Then we have $d_{\omega_{FS}}(p,\Phi_\infty(q_j))\to d_{\omega_{FS}} (p, \Phi_\infty(q_\infty)) = 0$ as $j\to \infty$, thus $p = \Phi_\infty(q_\infty)$ and this contradicts the choice of $p$.
\end{proof}

Since $\Phi_\infty$ is not surjective, there exists a $q'\in X_{can}\backslash \Phi_\infty(X_\infty)$.  Consider a point $q\in D\subset X$ with $\Phi(q) = q'$, for a sequence of points $\{q_j\}\subset \mathcal R_X=\mathcal R$ in the path constructed in Lemma \ref{finite distance} with $q_j\to q$, i.e. $d_{\omega_{FS}}(q',\Phi_\infty(q_j))\to 0$, we have 
\begin{equation*}
\sup_j d_\infty(p_\infty,q_j)<\infty.
\end{equation*}
This contradicts item (2) in Lemma \ref{cont lemma}. Hence the diameter of $(X,g(t_i))$ is uniformly bounded. And we finish the proof of Theorem \ref{thm:1}
\end{proof}

\begin{proof}[Proof of Corollary \ref{cor:1}]We will show $\Phi_\infty: X_\infty \to X_{can}$ is surjective. Suppose not, there is $p\not\in \Phi_\infty(X_\infty)$. Since $\Phi_\infty(\mathcal R)$ is dense in $X_{can}$, there exists a sequence of points $q_j\in \mathcal R$ such that $d_{\omega_{FS}} (p,\Phi_\infty(q_j))\to 0$ as $j\to\infty$. We have shown diam$(X_\infty, d_\infty)$ is bounded, hence $q_j$ would converge to some point $q_\infty\in X_\infty$ under $d_\infty$. Hence
\begin{equation*}
d_{\omega_{FS}} (p,\Phi_\infty(q_\infty))\ =  \lim_{j\to \infty} d_{\omega_{FS}}(p,\Phi_\infty(q_j)) = 0,
\end{equation*}
and we conclude that $p = \Phi_\infty(q_\infty)$, and thus a contradiction. Hence $\Phi_\infty$ is surjective. Combining with Song's result that $\Phi_\infty$ is also injective, we see that $\Phi_\infty$ is a Lipschitz continuous homeomorphism of $(X_\infty,d_\infty)$ and $X_{can}$, since $(X_\infty,d_\infty)$ is a compact space. Moreover, $\Phi_\infty|_{\mathcal R}:(\mathcal R,d_\infty)\to (\mathcal R_X, g_\infty)$ is an isometry so $\Phi_\infty$ induces an isometry between $(X_\infty, d_\infty)$ and $\overline{(\mathcal R_X, g_\infty)} = (X_{can}, g_\infty)$. Hence the Gromov-Hausdroff limit of the K\"ahler Ricci flow \eqref{KRF} is the canonical model of $X$, with the limit metric of the flow, under the assumption of bounded Ricci curvature along the flow.
\end{proof}

\begin{proof}[Proof of Theorem \ref{TZ}]
Suppose the flow \eqref{KRF} is of Type III, i.e. $|Rm|(g(t))$ is uniformly bounded,  by Shi's derivative estimates all derivatives of $Rm$ are bounded. Fix a point $p\in X\backslash D$, for any sequence $t_i\to \infty$, by the smooth convergence of $\omega(t_i)$ on $X\backslash D$ (Lemma \ref{lemma 1}), the volumes of unit balls $B_{g(t_i)}(p,1)\subset (X,g(t_i),p)$ are bounded below by a uniform positive constant, the limit space $(X_\infty,d_\infty,p_\infty)$ is smooth. Hence $X_\infty = \mathcal R$ and $\mathcal S= \mathcal S_X = \emptyset$. Then $\mathcal R_X = X$, otherwise, if there exists $q\in X\backslash \mathcal R_X$, then by Corollary  \ref{corollary finite} we have $d_{g(t_i)}(p,q)\le C_q$ for any $t_i$ and a uniform constant $C_q$ depending only on $q$, hence $q$ must converge to some point $q_\infty\in X_\infty$ along the Gromov-Hausdorff convergence, by the definition of $\mathcal S_X$, $q\in \mathcal S_X\neq \emptyset$, thus a contradiction. So we have $X$ is a compact K\"ahler manifold admitting  a smooth K\"ahler Einstein metric $\omega_{KE}$ with $\ric(\omega_{KE}) = -\omega_{KE}$, hence $K_X$ is ample.
\end{proof}

\section*{Appendix}

In this appendix, we will show that along the K\"ahler-Ricci flow \eqref{KRF}, assume Ricci curvature is uniformly bounded below for all $t\ge 0$, then for any sequence $t_i\to \infty$, $(X,\omega(t_i),p)$ is a sequence of almost K\"ahler-Einstein manifolds in the sense of Tian-Wang (\cite{TiWa}), where $p\in X\backslash D$ is a fixed point. Recall a sequence of K\"ahler manifolds $(X_i, \omega_i,p_i)$ is called almost K\"ahler Einstein if the following conditions are satisfied.
\begin{enumerate}[label=(\arabic*)]
\item $\ric(\omega_i)\ge -\omega_i$
\item $Vol_{\omega_i}(B(p_i,r_0))\ge v_0>0$, for two fixed constants $r_0>0$ and $v_0$.
\item The flow $\frac{\partial}{\partial t}\omega = -\ric(\omega) + \lambda_i \omega$ has a solution $\omega(t)$ with $\omega(0) = \omega_i$ on $X_i\times [0,1]$, where $\lambda_i\in [-1,1]$ is a constant. Moreover, $\int_0^1\int_{X_i}|R(\omega(t)) - n\lambda_i|\omega(t)^n dt\to 0$ as $i\to \infty$. 
\end{enumerate}

We may assume $\ric(\omega(t))\ge - K$ for a constant $K>0$ (we may assume $K\ge 1$) and any $t\ge 0$. Let $\tilde \omega_i = K\omega(t_i)$, then $\ric(\tilde \omega_i)\ge -1$. Since $(X,\omega(t_i))$ is non-collapsed at the point $p\in X\backslash D$ due to the smooth convergence, we have $(X,\tilde \omega_i)$ is also non-collapsed at $p$, i.e., there exists $v_0>0$ such that $Vol_{\tilde \omega_i}(B_{\tilde \omega_i} (p,r_0))\ge v_0$ for some small $r_0>0$.

$\tilde \omega_i(t) := K\omega(t_i + K^{-1} t)$ with $t\in[0,1]$ satisfies the (normalized) K\"ahler Ricci flow equation
\begin{equation*}
\frac{\partial}{\partial t}\tilde \omega_i(t) = -\ric(\tilde \omega_i(t)) - K^{-1}\tilde \omega_i(t),
\end{equation*}
with the initial $\tilde \omega_i( 0 ) = \tilde \omega_i$. From the evolution equation for the scalar curvature $R(\omega(t))$ ($\omega(t)$ is the solution to \eqref{KRF}) 
\begin{equation*}
\frac{\partial}{\partial t} R = \Delta_{\omega(t)} R + \abs{\ric} + R,
\end{equation*}
by maximum principle, at the minimum point of $R(\omega(t))$ for each $t$, $R_{\min} = \min_X R(\omega(t))$, we have
\begin{equation*}
\frac{d}{dt}R_{\min}(t)\ge \abs{\ric} + R_{\min}(t)\ge \frac{R_{\min}(t)^2}{n} + R_{\min}(t).
\end{equation*}
Standard comparison theorem of ODE implies that
\begin{equation*}
R_{\min}(t)\ge - n -\frac{R_{\min}(0)n + n^2}{R_{\min}(0) e^t - R_{\min}(0) - n} \ge  -n - O(e^{-t}).
\end{equation*}
Hence for $t\in [0,1]$, we have
\begin{equation*}
R(\tilde \omega_i(t)) = K^{-1} R(\omega(t_i + K^{-1}t))\ge - K^{-1}n - O(e^{-t_i}).
\end{equation*}
Then
\begin{align*}
\int_0^1\int_X |R(\tilde \omega_i(t)) + K^{-1}n| \tilde\omega_i(t)^n dt &\le \int_0^1\int_X \big((R(\tilde \omega_i(t)) + K^{-1}n) + O(e^{-t_i})\big) \tilde\omega_i(t)^n dt \\ &
 = \int_0^1 \int_X n (\ric(\tilde \omega_i(t)) +K^{-1} \tilde \omega_i(t))\wedge \tilde \omega_i(t)^{n-1} dt + O(e^{-t_i})\\
 & = \int_0^1 \int_X n(e^{-t_i - K^{-1}t} (\omega_0 - \chi) - \ddbar \dot\varphi)\wedge \tilde\omega_i(t)^{n-1} dt + O(e^{-t_i})\\
 & = \int_0^1 \int_X n e^{-t_i - K^{-1}t} (\omega_0 - \chi)\wedge \tilde\omega_i(t)^{n-1} dt + O(e^{-t_i})\\
 &\le O(e^{-t_i})\to 0,\quad \text{as }t_i\to \infty.
\end{align*}

\end{document}